\documentclass[a4paper]{amsart}
\usepackage{amsmath,amsthm,amssymb,latexsym,epic,bbm,comment,mathrsfs}
\usepackage{graphicx,enumerate,stmaryrd,xcolor}
\usepackage[all,2cell,knot]{xy}
\xyoption{2cell}

\newtheorem{theorem}{Theorem}
\newtheorem{lemma}[theorem]{Lemma}
\newtheorem{corollary}[theorem]{Corollary}
\newtheorem{proposition}[theorem]{Proposition}

\newtheorem{example}[theorem]{Example}

\usepackage[all]{xy}
\usepackage[active]{srcltx}
\usepackage[parfill]{parskip}
\usepackage{enumerate}

\newcommand{\tto}{\twoheadrightarrow}

\begin{document}

\title[Hybrid KL basis and parabolic induction]
{Hybrid Kazhdan-Lusztig basis and parabolic induction}
\author[V.~Mazorchuk and S.~Srivastava]{Volodymyr 
Mazorchuk and Shraddha Srivastava}

\begin{abstract}
We provide a categorification for the hybrid 
Kazhdan-Lusztig bases of the Hecke algebra by
constructing and investigating  a series of 
new stratified structures  on the principal 
block of BGG category $\mathcal{O}$.
\end{abstract}

\maketitle

\section{Introduction and description of the results}\label{s0}

\subsection{Background}\label{s0.1}

Kazhdan-Lusztig combinatorics of the Hecke algebra of 
a finite Weyl group, which originates from \cite{KL}, 
is an intensively studied very useful tool with far
reaching applications in various areas, notably, 
representation theory. For instance, this combinatorics
describes various composition multiplicities for structural
modules in the principal block of the BGG category 
$\mathcal{O}$ associated to a triangular decomposition
of a semi-simple finite dimensional complex Lie algebra, 
see \cite{BGG,Hu}.

Arguably, the most famous ingredient of Kazhdan-Lusztig 
combinatorics is the {\em Kazhdan-Lusztig basis} of the 
Hecke algebra, see \cite{KL}, which has remarkable 
positivity properties. The coefficients of the transformation
matrix between the standard and the Kazhdan-Lusztig bases
are called {\em Kazhdan-Lusztig polynomials}. The coefficients
of these polynomials determine, among other things, 
(graded) composition multiplicities of simple subquotients 
of Verma modules, see \cite{KL,Hu,BB81,BK81}. 

The connection between the Kazhdan-Lusztig basis
and the principal block of 
category $\mathcal{O}$ for a Lie algebra $\mathfrak{g}$
is given by the Grothendieck decategorification of the latter.
In more detail, let $\mathcal{O}_0$ be the principal block of 
$\mathcal{O}$. Then $\mathcal{O}_0$ admits a 
$\mathbb{Z}$-graded lift ${}^{\mathbb{Z}}\mathcal{O}_0$ and 
the Grothendieck group of ${}^{\mathbb{Z}}\mathcal{O}_0$ is
naturally isomorphic to the Hecke algebra of the Weyl group of
$\mathfrak{g}$. The Grothendieck group of ${}^{\mathbb{Z}}\mathcal{O}_0$
has three natural bases: the basis of simple modules, the basis of 
Verma modules and the basis of indecomposable projective modules.
It turns out that, if we fix an isomorphism between the 
Grothendieck group of ${}^{\mathbb{Z}}\mathcal{O}_0$ 
and the Hecke algebra sending Verma modules
to the standard basis, this isomorphism will match indecomposable
projective modules with the Kazhdan-Lusztig basis. Simple modules
will then be matched with the so-called {\em dual Kazhdan-Lusztig basis}.
And the positivity properties of the Kazhdan-Lusztig basis are then
explained by the fact that indecomposable projective modules in 
$\mathcal{O}_0$ have Verma flags, that is, filtrations with Verma
subquotients. In the language of \cite{CPS0}, this means that 
$\mathcal{O}_0$ is a {\em highest weight category}, alternatively, that
the associative algebra whose module category is equivalent to 
$\mathcal{O}_0$ is {\em quasi-hereditary}, see \cite{DR}.

\subsection{Motivation}\label{s0.2}

The recent preprint \cite{BMS} investigates a series of the so-called
{\em hybrid Kazhdan-Lusztig bases} of the Hecke algebra, defined in
\cite{GH07}, and establishes for them a number of interesting 
positivity results. The datum necessary to define a hybrid basis
consists of the choice of a parabolic subgroup of the Weyl group and
the choice of a side (left/right). The elements of the hybrid basis 
are certain products of the elements of the Kazhdan-Lusztig basis for
this parabolic subgroup with certain elements of the standard basis
for the whole group (these latter elements are indexed by  shortest 
coset representatives of the group modulo the subgroup).
The choice of a side reflects 
the choice of left or right cosets, alternatively, the order of 
multiplication of the factors. The two extreme choices of parabolic 
subgroups give the Kazhdan-Lusztig basis and the standard basis, 
respectively.

The paper \cite{BMS} establishes positivity for the coefficients 
of the transformation matrix between two hybrid bases if one of
the corresponding parabolic subgroups is contained in the other one. 
This naturally leads to the following question: does this
positivity phenomenon have a categorical interpretation? This 
question is our principal motivation for the present paper.

\subsection{Results}\label{s0.3}

In the present paper we give a complete positive answer to our
main motivation question. Here is a brief summary of our
main results:
\begin{itemize}
\item Given a parabolic subgroup of the Weyl group and a choice of
the left/right side, we show that ${}^{\mathbb{Z}}\mathcal{O}_0$
has a stratified structure, in the sense of \cite{CPS,Fr}, such that
the standard modules with respect to this structure match with the 
elements of the corresponding hybrid Kazhdan-Lusztig basis,
see Theorems~\ref{thm-main-1} and \ref{thm-main2}. 
\item We describe titling modules with respect to such a stratified
structure, see Propositions~\ref{prop-tilting1} and \ref{prop-tilting2}.
\item We determine natural duals for the hybrid Kazhdan-Lusztig basis,
see Subsection~\ref{s3.5}.
\item As an auxiliary technical result, we discover new relations between 
projective and shuffling functors, see Lemma~\ref{lem-tilt2-1}.
\end{itemize}

One of the axioms of a stratified structure is that projective objects
must have a filtration with standard subquotients. Given one parabolic
subgroup that is a subgroup of another parabolic subgroup, we show
that the standard objects for the second structure have a filtration
whose subquotients are standard modules for the first structure,
see Corollaries~\ref{cor-s3.3-5} and \ref{cor-s4.3-57}. This provides
a categorical explanation of the positivity results of \cite{GH07,BMS}.

\subsection{Methods}\label{s0.4}

We use a wide variety of different methods and techniques, mostly of algebraic,
combinatorial and homological nature. Several of our results have
unexpectedly involved proofs. To start with, it is worth to point
out that a choice of the left/right side is not at all symmetric.
In fact, at the combinatorial level, it is symmetric, but, at the
categorical level, it is not. The reason for that is the observation
in \cite[Remark~1.2]{MS2} that $\mathcal{O}_0$ does not have a self-equivalence
which induces the map $w\mapsto w^{-1}$, where $w$ is an element of
the Weyl group, on the indices of the isomorphism classes of simple
objects.

In the case when our parabolic subgroup acts on the left, we 
construct our stratified structure using parabolic induction,
in the spirit of \cite{FKM}. Here are approach and arguments are
natural adaptations of those in \cite{FKM}. An additional 
simplification in this case is that all constructions are naturally
adjusted to the action of projective functors.

In the case when our parabolic subgroup acts on the right, none of 
these are available and we need to employ various tricks. For example,
we use the equivalence from \cite{So86} to transfer 
proper standard modules from the other side. However, the major 
problem arises when trying to determine the tilting modules for
this structure. On the other side, this is not too difficult as
we can use the action of projective functors, just like in the 
classical case. But here, this is not available and we were forces to
go in a roundabout way, using certain equivalences for Serre 
subquotients established in \cite{CMZ} as well as various ideas
from $\nabla$-filtration dimension of modules which go back to
\cite{Pa,MO}.

Our most annoying proof is that of a technical auxiliary lemma 
that establishes certain relations in the singular braid monoid
of our Weyl group. The relevance of this to the present paper 
is the observation of \cite{MS2,J-Z} that it is this monoid
that acts on $\mathcal{O}_0$ via shuffling and projective functors.
Out wannabe tilting objects are defined using shuffling functors
and our standard objects are defined using projective functors.
This is where all this comes from. For our proof, we need the 
property that the conjugation with the longest element in the 
Weyl group rearranges the singular generators of
the singular braid monoid in the same way
as it rearranges the simple reflections of the Weyl group. 
Unfortunately, we could not find any elegant argument to prove this
and hence were forced to do a brute force case-by-case walk-through.
As one can expect, types $E_7$ and, especially, $E_8$ are not at all pretty
and require a lot of space, as well as computations assistance of 
SageMath.

\subsection{Structure}\label{s0.5}

The paper is organized as follows. 
In Section~\ref{s1}, we collected all necessary preliminaries about 
category $\mathcal{O}$ and its combinatorics.
Section~\ref{s2} goes into detail of parabolic induction, which is 
our main technical tool for construction of stratified structures
in the case when the parabolic subgroup acts on the left.
This case is dealt with in Section~\ref{s3}, which starts with some 
general preliminaries on stratified algebras. The main result here
is Theorem~\ref{thm-main-1} in Subsection~\ref{s3.2} which 
asserts existence of a stratified structure. Many other 
structural properties of this structure are given in 
subsections that follow. In particular, tilting modules are described
in Proposition~\ref{prop-tilting1} which can be found in 
Subsection~\ref{s3.6}.

The case of the parabolic subgroup acting on the right is considered 
in Section~\ref{s4}, where the main result is 
Theorem~\ref{thm-main2} in  Subsection~\ref{s4.2}.
Again, many other  structural properties of the
corresponding stratified structure are given in 
subsections that follow. In particular, tilting modules are described
in Proposition~\ref{prop-tilting2} which can be found in 
Subsection~\ref{s4.6}. One of the auxiliary statement here,
namely, Lemma~\ref{lem-tilt2-1}, is given without a proof.
In fact, this proof is postponed until Section~\ref{s5},
where it is given with all details, however, using a 
brute force case-by-case analysis of all finite irreducible
root systems. As one could expect in such case, type $E_7$
and, especially, type $E_8$ look quite ugly.

\subsection*{Acknowledgements}
The first author is partially supported by the Swedish Research Council.
All computations were done using SageMath. We used Microsoft Copilot
as a general help for SageMath coding.
\vspace{5mm}

\section{Category $\mathcal{O}$ and its combinatorics}\label{s1}

\subsection{Setup}\label{s1.1}

In this paper, we work over $\mathbb{C}$. For a vector space $V$,
we denote by $V^*$ the dual space $\mathrm{Hom}_\mathbb{C}(V,\mathbb{C})$.

Let $\mathfrak{g}$ be a 
semi-simple finite dimensional Lie algebra with a fixed triangular 
decomposition
\begin{equation}\label{eq-01}
\mathfrak{g}=\mathfrak{n}_-\oplus \mathfrak{h}\oplus \mathfrak{n}_+. 
\end{equation}
Here $\mathfrak{h}$ is a fixed Cartan subalgebra. Let $\mathbf{R}$
be the root system of the pair $(\mathfrak{g},\mathfrak{h})$.
For $\alpha\in \mathbf{R}$, we fix a non-zero element
$X_\alpha$ in the corresponding root space $\mathfrak{g}_\alpha$.
Let $W$ be the Weyl group of $\mathbf{R}$. Let $\pi$ be the basis of
$\mathbf{R}$ which corresponds to the triangular decomposition
\eqref{eq-01} and $\mathbf{R}=\mathbf{R}_-\cup \mathbf{R}_+$ the corresponding
decomposition of $\mathbf{R}$ into negative and positive roots.
We denote by $\rho$ the half of the sum of all positive roots.
We also consider the dot-action of $W$ on $\mathfrak{h}^*$
defined via $w\cdot \lambda:=w(\lambda+\rho)-\rho$, for
$\lambda\in \mathfrak{h}^*$ and $w\in W$. As usual, we denote
by $w_0$ the longest element of $W$.

We let $S$ denote the set of all simple reflections in $W$ which 
correspond to the basis $\pi$. For any subset $J\subset S$,
we denote by $W(J)$ the parabolic subgroup of $W$ which is generated
by all reflections in $J$. In particular, $W(S)=W$.
We denote by $W_J$ and ${}_JW$ the sets of shortest representatives
in the cosets from $W/_{W(J)}$ and 
${}_{W(J)}\hspace{-1mm}\setminus\hspace{-1mm} W$, respectively.
We denote by $w_0^{J}$ the longest element in $W(J)$.

Denote by $U(\mathfrak{g})$ the universal enveloping algebra of 
$\mathfrak{g}$.
Recall that a $\mathfrak{g}$-module $M$ is called a {\em weight module}
provided that
\begin{displaymath}
M=\bigoplus_{\lambda\in\mathfrak{h}^*}M_\lambda, \,
\text{ where } \,
M_\lambda:=\{m\in M\,:\, h\cdot m =\lambda(h)m\text{ for all }
h\in\mathfrak{h}\}.
\end{displaymath}
The support $\mathrm{supp}(M)$ of a weight module $M$ is the set of all
$\lambda\in \mathfrak{h}^*$ such that $M_\lambda\neq 0$.
Another way to describe a weight module is to say that the 
action of $\mathfrak{h}$ on such a module is diagonalizable.

\subsection{Category $\mathcal{O}$}\label{s1.2}

Consider the Bernstein-Gelfand-Gelfand category 
$\mathcal{O}=\mathcal{O}(\mathfrak{g})$
associated to the triangular decomposition
\eqref{eq-01}, see \cite{BGG,Hu}. This category consists of all
finitely generated weight $U(\mathfrak{g})$-modules the action of 
$U(\mathfrak{n}_+)$ on which is locally finite.

Simple objects in $\mathcal{O}$ are, up to isomorphism, 
the simple highest weight modules $L(\lambda)$, where
$\lambda\in\mathfrak{h}^*$ is the highest weight of
$L(\lambda)$. We denote by $\Delta(\lambda)$ the Verma
module with highest weight $\lambda$, see \cite{Di,Hu}.
We denote by $P(\lambda)$ and $I(\lambda)$ the 
indecomposable projective cover and injective envelope
of $L(\lambda)$, respectively. We also denote by
$T(\lambda)$ the indecomposable tilting module with 
highest weight $\lambda$, see \cite{CI,Ri,Hu}.

We denote by $\star$ the contravariant duality on
$\mathcal{O}$ which preserves the isomorphisms classes of
simple modules, that is $L(\lambda)^\star\cong L(\lambda)$,
for all $\lambda\in\mathfrak{h}^*$. Then we also have
$P(\lambda)\cong I(\lambda)^\star$. Using $\star$,
we can define the {\em dual Verma modules}
$\nabla(\lambda):=\Delta(\lambda)^\star$, 
for $\lambda\in\mathfrak{h}^*$. 

In case several algebras are around, to avoid confusion we will
sometimes use the algebra as a subscript in the notation of 
a module. For example, we can write $\Delta_\mathfrak{g}(\lambda)$
to specify that this is the Verma module $\Delta(\lambda)$
over the algebra $\mathfrak{g}$.

\subsection{Principal block}\label{s1.3}

Consider the principal block $\mathcal{O}_0$ of 
$\mathcal{O}$, that is, the direct summand containing the
trivial $\mathfrak{g}$-module $L(0)$. Simple objects
in $\mathcal{O}_0$ are bijectively indexed by
the elements in $W$ as follows: $L_w:=L(w\cdot 0)$.
We will then also similarly simplify the remaining notation
for the structural modules in $\mathcal{O}_0$:
$\Delta_w:=\Delta(w\cdot 0)$, $\nabla_w:=\nabla(w\cdot 0)$, 
$P_w:=P(w\cdot 0)$, $I_w:=I(w\cdot 0)$ and $T_w:=T(w\cdot 0)$.

We denote by $A$ the (unique up to isomorphism) basic
finite dimensional associative algebra for which 
the category $A$-mod of finite dimensional modules
is equivalent to $\mathcal{O}_0$. The algebra 
$A$ is Koszul, see \cite{So}, so we can fix a 
positive Koszul $\mathbb{Z}$-grading on $A$ as follows:
\begin{displaymath}
A\cong\bigoplus_{i\in\mathbb{Z}}A_i. 
\end{displaymath}

\subsection{Graded lift}\label{s1.4}

We denote by ${}^{\mathbb{Z}}\mathcal{O}_0$ the category of
finite-dimensional graded $A$-modules
\begin{displaymath}
M\cong\bigoplus_{i\in\mathbb{Z}}M_i. 
\end{displaymath}
Morphisms in 
${}^{\mathbb{Z}}\mathcal{O}_0$ are homogeneous maps of
degree $0$. We denote by $\langle 1\rangle$ the 
auto-equivalence of ${}^{\mathbb{Z}}\mathcal{O}_0$
which shifts the grading as follows:
$(M\langle 1\rangle)_i=M_{i+1}$, for $i\in\mathbb{Z}$
and $M\in {}^{\mathbb{Z}}\mathcal{O}_0$.
We have the obvious functor $\mathbf{F}$ from 
${}^{\mathbb{Z}}\mathcal{O}_0$ to
$\mathcal{O}_0$ which forgets the grading.
We will say that $X\in \mathcal{O}_0$
{\em admits graded lift} provided that there exists
$Y\in {}^{\mathbb{Z}}\mathcal{O}_0$ such that 
$X\cong\mathbf{F}(Y)$. If an indecomposable 
$X\in \mathcal{O}_0$ admits graded lift, then this
lift is unique up to isomorphism and shift of grading, 
see \cite{St}.

All structural modules in $\mathcal{O}_0$ admit graded lifts,
see \cite{MO3}.
Given $w\in W$, we denote the graded lift of 
$L_w$ concentrated in degree $0$ also by $L_w$,
abusing notation for simplicity. Similarly,
the graded lifts of $\Delta_w$ and $P_w$ which
cover $L_w$ in the graded sense are denoted
by the same respective symbols. Next, the graded
lifts of $I_w$ and $\nabla_w$ which envelop 
$L_w$ in the graded sense are denoted by the 
same respective symbol. Finally, the graded 
lift of $T_w$ which envelops $\Delta_w$
(or, equivalently, covers $\nabla_w$) in the 
graded sense is denoted by the same symbol.

\subsection{Hecke algebra}\label{s1.5}

Let $\mathbb{A}:=\mathbb{Z}[v,v^{-1}]$.
Let $\mathbf{H}$ be the Hecke algebra of $(W,S)$.
The algebra $\mathbf{H}$ is an $\mathbb{A}$-algebra
with the standard basis $\{H_w\,:\,w\in W\}$ and
the multiplication determined by the braid relations
in $W$ as well as the following quadratic relations
for the generators $H_s$, where $s\in S$:
\begin{displaymath}
H_s^2=(v^{-1}-v)H_s+H_e. 
\end{displaymath}
We denote by $\{\underline{H}_w\,:\,w\in W\}$
the Kazhdan-Lusztig basis of $\mathbf{H}$, see \cite{KL}.
Note that we are using the normalization of \cite{So1}.
The entries $\{p_{y,x}\,:\,x,y\in W\}$
of the transformation matrix between these two bases
are called Kazhdan-Lusztig polynomials:
\begin{displaymath}
\underline{H}_x=\sum_{y\in W}p_{y,x}H_y,\quad x,y\in W.
\end{displaymath}
We note that $p_{y,x}\neq 0$ implies that
$y\leq x$ with respect to the Bruhat order on $W$.

We also denote by $\{\underline{\hat{H}}_w\,:\,w\in W\}$
the dual Kazhdan-Lusztig basis of $\mathbf{H}$, defined
via $\tau(\underline{\hat{H}}_x\underline{H}_{y^{-1}})=\delta_{x,y}$,
the Kronecker symbol. Here $\tau$ is the standard trace on $\mathbf{H}$
given by $\tau(H_w)=\delta_{w,e}$.

\subsection{Projective functors}\label{s1.6}

For $w\in W$, we denote by $\theta_w$ the indecomposable projective
endofunctor of $\mathcal{O}_0$ which is defined uniquely, up to 
isomorphism, via $\theta_w P_e\cong P_w$, see \cite{BG}.
We denote by $\mathscr{P}$ the monoidal category of projective
endofunctors of $\mathcal{O}_0$.

Each $\theta_w$ admits a graded lift which we denote by the same
symbol, see \cite{St}. Then $\theta_w P_e\cong P_w$, now in the graded sense.
We denote by ${}^\mathbb{Z}\mathscr{P}$ the monoidal category of 
graded projective endofunctors of ${}^\mathbb{Z}\mathcal{O}_0$.

\subsection{Connection}\label{s1.7}

Consider the Grothendieck group 
$\mathbf{Gr}({}^\mathbb{Z}\mathcal{O}_0)$ of 
${}^\mathbb{Z}\mathcal{O}_0$. Fix an isomorphism between 
$\mathbf{Gr}({}^\mathbb{Z}\mathcal{O}_0)$ and
$\mathbf{H}$ by sending $[\Delta_w]$, where $w\in W$,
to $H_w$, and interpreting $v$ as $\langle -1\rangle$.
Under this isomorphism, $[P_w]$ is sent to 
$\underline{H}_w$ and $[L_w]$ is sent to $\underline{\hat{H}}_w$.

Consider the split Grothendieck ring
$\mathbf{Gr}_{\oplus}({}^\mathbb{Z}\mathscr{P})$ of 
${}^\mathbb{Z}\mathscr{P}$. Fix an isomorphism between 
$\mathbf{Gr}_{\oplus}({}^\mathbb{Z}\mathscr{P})$ and
to $\mathbf{H}$ by sending $[\theta_w]$, where $w\in W$,
to $\underline{H}_w$, and interpreting $v$ as $\langle -1\rangle$.

As all projective functors are exact, the group 
$\mathbf{Gr}({}^\mathbb{Z}\mathcal{O}_0)$ is naturally
a $\mathbf{Gr}_{\oplus}({}^\mathbb{Z}\mathscr{P})$-module.
Under the above isomorphisms, this corresponds to the right
regular $\mathbf{H}$-module $\mathbf{H}_\mathbf{H}$.

Note that, for $M,N\in {}^\mathbb{Z}\mathcal{O}_0$, we have
\begin{equation}\label{bilinearform}
([M],[N]):=\tau([M][N])=\sum_{i,j\in\mathbb{Z}}
(-1)^iv^j\dim\mathcal{D}^b({}^\mathbb{Z}\mathcal{O}_0)
(M[-i]\langle j\rangle,N^\star),
\end{equation}
as both are linear in $M$ and $N$ and agree when $M$ and $N$ 
independently run through all Verma modules
(see also \cite[Subsection~3.4]{CM} and \cite{KM}).

\subsection{Other regular integral blocks}\label{s1.8}

Recall that a weight $\lambda\in\mathfrak{h}^*$ is called
{\em regular} if its dot-stabilizer is trivial, that is,
$|\{w\cdot \lambda\,:\,w\in W\}|=|W|$. Also, a weight
$\lambda\in\mathfrak{h}^*$ is called {\em integral}
if it appears in the support of some finite dimensional
$\mathfrak{g}$-module. A weight is called {\em dominant}
provided that $\Delta(\lambda)=P(\lambda)$.
Note that regular, integral, dominant weights are
exactly the highest weights of simple finite dimensional
$\mathfrak{g}$-modules.

For a regular, integral, dominant weight $\lambda\in\mathfrak{h}^*$, we
can consider the corresponding block $\mathcal{O}_\lambda$
of $\mathcal{O}$ defined as the direct summand containing 
$L(\lambda)$. Alternatively, $\mathcal{O}_\lambda$ is the 
Serre subcategory of $\mathcal{O}$ generated by all
$L(w\cdot\lambda)$, where $w\in W$. The principal
block $\mathcal{O}_0$ is just a special case of this,
for $\lambda=0$. The blocks 
$\mathcal{O}_\lambda$ and $\mathcal{O}_0$ are equivalent
via an equivalence which matches $L(w\cdot\lambda)$ with
$L_w$, for $w\in W$.

\section{Parabolic induction and restriction}\label{s2}

\subsection{Parabolic subalgebras}\label{s2.1}

Recall that $S$ denotes the set of all simple reflections in $W$ which 
correspond to the basis $\pi$. For any subset $J\subset S$, we
have the corresponding parabolic subalgebra $\mathfrak{p}=\mathfrak{p}(J)$
of $\mathfrak{g}$. The subalgebra $\mathfrak{p}$ is generated by
$\mathfrak{h}$ and $\mathfrak{n}_+$ as well as by all root subspaces
$\mathfrak{g}_\alpha$, where $\alpha$ is a negative root such that 
$-\alpha$ is a simple root and the reflection with respect to 
$-\alpha$ belongs to $J$. We write
\begin{displaymath}
\mathfrak{p}=\mathfrak{a}\oplus \mathfrak{h}^{\perp}\oplus \mathfrak{n}, 
\end{displaymath}
where 
\begin{itemize}
\item $\mathfrak{a}$ is the semi-simple Levi subalgebra;
\item $\mathfrak{h}^{\perp}\oplus \mathfrak{n}$ is the radical
of $\mathfrak{p}$;
\item $\mathfrak{n}=[\mathfrak{h}^{\perp}\oplus \mathfrak{n},
\mathfrak{h}^{\perp}\oplus \mathfrak{n}]$ is nilpotent;
\item $\mathfrak{h}^{\perp}\subset \mathfrak{h}$ commutes with
$\mathfrak{a}$.
\end{itemize}
We note that 
$\mathfrak{n}$ is the direct sum of all 
$\mathfrak{g}_\alpha$, where $\alpha$ is a positive root which is
not in the linear span of $J$, that
$\mathfrak{a}\oplus \mathfrak{h}^{\perp}$ is reductive
and that $\mathfrak{h}^{\perp}$ is the center of 
$\mathfrak{a}\oplus \mathfrak{h}^{\perp}$. We denote by
$\mathfrak{h}_\mathfrak{a}$ the intersection of $\mathfrak{a}$
with $\mathfrak{h}$. Then we have 
$\mathfrak{h}=\mathfrak{h}_\mathfrak{a}\oplus\mathfrak{h}^{\perp}$.
We also note that, intersecting $\mathfrak{a}$ with the components of the
triangular decomposition for $\mathfrak{g}$, gives a triangular decomposition
for $\mathfrak{a}$.

\begin{example}\label{ex-s2.1-1}
{\em Consider $\mathfrak{g}=\mathfrak{sl}_3$ with standard basis
$e_{12}$, $e_{13}$, $e_{23}$, $e_{21}$, $e_{31}$, $e_{32}$
(the matrix units) as well as $h_1:=e_{11}-e_{22}$
and $h_2=e_{22}-e_{33}$. Then, for the standard triangular
decomposition, we have that $\{e_{12},e_{13},e_{23}\}$
is a basis of $\mathfrak{n}_+$, $\{e_{21},e_{31},e_{32}\}$
is a basis of $\mathfrak{n}_-$ and $\{h_1,h_2\}$ is a
basis of $\mathfrak{h}$.

Let $\alpha$ be the simple root corresponding to $e_{12}$
and $\beta$ be the simple root corresponding to $e_{23}$.
Then, $\pi=\{\alpha,\beta\}$ and, for the subset 
$\{\alpha\}\subset\pi$, the corresponding parabolic
subalgebra $\mathfrak{p}$ contains $\mathfrak{n}_+$
and $\mathfrak{h}$ and has one extra dimension, given
by the additional basis vector $e_{21}$. In this case
$\mathfrak{a}$ has basis $\{e_{12},e_{21},h_1\}$
and is isomorphic to $\mathfrak{sl}_2$. Note that
$\mathfrak{h}_\mathfrak{a}$ has basis $h_1$.
Further,
$\mathfrak{n}$ has basis $\{e_{13},e_{23}\}$. Finally,
$\mathfrak{h}^{\perp}$ is generated by $h_1+2h_2$
(this latter element equals $e_{11}+e_{22}-2e_{33}$
and hence commutes with both $e_{12}$ and $e_{21}$).
\hfill$\blacksquare$
}
\end{example}

\subsection{Parabolic induction}\label{s2.2}

Let $\mathfrak{p}$ be a parabolic subalgebra of $\mathfrak{g}$
as in Subsection~\ref{s2.1}. 

Given an $\mathfrak{a}$-module $M$ and $\eta\in(\mathfrak{h}^{\perp})^*$,
we can set $\mathfrak{n}\cdot M=0$ as well as
$h\cdot m:=\eta(h)m$, for $h\in \mathfrak{h}^{\perp}$
and $m\in M$, and in this way turn $M$
into a $\mathfrak{p}$-module. After that, we can consider the 
induced $\mathfrak{g}$-module
\begin{displaymath}
U(\mathfrak{g})\bigotimes_{U(\mathfrak{p})}M.
\end{displaymath}
The obvious assignment on morphisms defines a functor
\begin{displaymath}
\mathrm{Ind}_{\eta}:\mathfrak{a}\text{-}\mathrm{Mod}\to
\mathfrak{g}\text{-}\mathrm{Mod},
\end{displaymath}
called {\em parabolic induction}. Note that $\mathrm{Ind}_{\eta}$
depends on the choice of both $\mathfrak{p}$ and $\eta$.
Also note that $\mathrm{Ind}_{\eta}$ is exact due to the 
PBW Theorem.

Given $\nu\in(\mathfrak{h}_\mathfrak{a})^*$ and $\eta$ as above,
we can use $\mathfrak{h}=\mathfrak{h}_\mathfrak{a}\oplus\mathfrak{h}^{\perp}$
to define a unique $\mu(\nu,\eta)\in \mathfrak{h}$ whose restrictions to the 
two summands are given by $\nu$ and $\eta$, respectively.

\begin{lemma}\label{lem-s2.2-2}
For $\nu$ and $\eta$ as above, we have
\begin{displaymath}
\Delta_{\mathfrak{g}}(\mu(\nu,\eta)) \cong 
\mathrm{Ind}_{\eta}(\Delta_{\mathfrak{a}}(\nu)).
\end{displaymath}
\end{lemma}

\begin{proof}
This follows directly from the transitivity of induction. 
\end{proof}

\begin{corollary}\label{cor-s2.2-3}
For any $\eta$ as above, the functor
$\mathrm{Ind}_{\eta}$ maps $\mathcal{O}(\mathfrak{a})$
to $\mathcal{O}(\mathfrak{g})$.
\end{corollary}

\begin{proof}
In Lemma~\ref{lem-s2.2-2}, we already established that 
$\mathrm{Ind}_{\eta}$ sends Verma modules to Verma modules.
From exactness, it follows that simple highest weight 
$\mathfrak{a}$-modules are sent by $\mathrm{Ind}_{\eta}$
to $\mathcal{O}(\mathfrak{g})$. Also, from the definition of
$\mathrm{Ind}_{\eta}$, we see that it sends 
weight $\mathfrak{a}$-modules to weight 
$\mathfrak{g}$-modules. As $\mathcal{O}(\mathfrak{g})$ 
is closed under (first) extensions in the category of 
weight $\mathfrak{g}$-modules, the claim follows
from the exactness of $\mathrm{Ind}_{\eta}$.
\end{proof}

If $\nu\in(\mathfrak{h}_\mathfrak{a})^*$ is regular, integral
and dominant,
then the $\mathfrak{g}$-modules
$\mathrm{Ind}_\eta(L_\mathfrak{a}(\nu))$
are known as {\em parabolic Verma modules}
or {\em generalized Verma modules}. They are exactly the 
standard objects in the parabolic subcategory $\mathcal{O}^\mathfrak{p}$
of $\mathcal{O}$, see \cite{RC}. These modules were intensively studied,
see \cite{Le77,Bo85,CC87,CIS,Mat} and references therein.
More generally, for an arbitrary simple
$\mathfrak{a}$-module $L$, the $\mathfrak{g}$-module 
$\mathrm{Ind}_\eta(L_\mathfrak{a})$ is referred to as a 
{\em generalized Verma module}. Such modules were also intensively studied,
see, e.g., \cite{Fu,Ma95,Ma98,MO,KhM99,Ma00,BFL,MS} and references therein.

\subsection{Restriction}\label{s2.3}

Now let us assume that we are given $M\in\mathcal{O}(\mathfrak{g})$
and $\eta\in(\mathfrak{h}^\perp)^*$. Then the subspace
\begin{displaymath}
\mathrm{Res}_\eta(M):=\{m\in M\,: h\cdot m=\eta(h)m\text{ for all }
h\in \mathfrak{h}^\perp\} 
\end{displaymath}
is an $\mathfrak{a}$-submodule of $M$, since $\mathfrak{a}$
commutes with $\mathfrak{h}^\perp$.

\begin{lemma}\label{lem-s2.3-1}
We have $\mathrm{Res}_\eta(M)\in \mathcal{O}(\mathfrak{a})$. 
\end{lemma}

\begin{proof}
That  $\mathrm{Res}_\eta$ sends weight $\mathfrak{g}$-modules 
to weight $\mathfrak{a}$-modules is clear from the definitions.
Also, that $\mathrm{Res}_\eta$ sends $U(\mathfrak{n}_+)$-
finite $\mathfrak{g}$-modules to $\mathfrak{a}$-modules,
the action of $U(\mathfrak{n}_+)\cap U(\mathfrak{a})$ on which 
is locally finite, is also clear. 

It remains to check that $\mathrm{Res}_\eta(M)$ is finitely
generated, for any $M\in \mathcal{O}(\mathfrak{g})$. Since
all modules in $\mathcal{O}(\mathfrak{g})$ are weight modules,
the assignment $M\mapsto \mathrm{Res}_\eta(M)$ is, obviously,
exact. As any object in $\mathcal{O}(\mathfrak{g})$ has finite
length, it follows that it is enough to prove out claim in the case
$M=L(\lambda)$. As $\Delta(\lambda)$ surjects onto $L(\lambda)$,
it is enough to prove out claim in the case $M=\Delta(\lambda)$.

By construction, the module $\Delta(\lambda)$ is free over
$U(\mathfrak{n}_-)$ of rank $1$. We claim that 
$\mathrm{Res}_\eta(\Delta(\lambda))$ is free over
$U(\mathfrak{n}_-)\cap U(\mathfrak{a})$ of finite rank.

Let us note that the module $\mathrm{Res}_\eta(\Delta(\lambda))$ 
is zero for almost all $\eta$, in which case we have nothing to prove.
Therefore let us assume that $\eta$ is such that 
$\mathrm{Res}_\eta(\Delta(\lambda))\neq 0$.
%
Let $\mathbf{p}:\mathfrak{h}^*\to (\mathfrak{h}^\perp)^*$
be the projection map with respect to the decomposition
\begin{displaymath}
\mathfrak{h}^*\cong \mathfrak{h}_\mathfrak{a}^*\oplus 
(\mathfrak{h}^\perp)^*.
\end{displaymath}
Then
\begin{displaymath}
\mathrm{Res}_\eta(\Delta(\lambda))=
\bigoplus_{\mu\in\mathfrak{h}^*\,:\,\mathbf{p}(\mu)=\eta}\Delta(\lambda)_\mu.
\end{displaymath}
Given a negative root $\alpha$ of $\mathfrak{g}$, we note that 
$\alpha$ is a root of $\mathfrak{a}$ if and only if 
$\mathbf{p}(\alpha)=0$. Let $\alpha_1,\dots,\alpha_k$ be a 
complete and irredundant
list of all negative roots of $\mathfrak{g}$ that are not 
roots of $\mathfrak{a}$. Then all $\mathbf{p}(\alpha_i)$ are non-zero,
moreover, the additive semigroup generated by all these elements in 
$(\mathfrak{h}^\perp)^*$ is free over the basis given by
those $\mathbf{p}(\alpha_i)$'s, for which $-\alpha_i$ is a simple root.
Therefore the equation
\begin{equation}\label{eq-gvm}
x_1\mathbf{p}(\alpha_1)+x_2\mathbf{p}(\alpha_2)+\dots + 
x_k\mathbf{p}(\alpha_k)  =\eta-\mathbf{p}(\lambda)
\end{equation}
has only finitely many non-negative integer solutions
$(x_1,x_2,\dots,x_k)\in\mathbb{Z}_{\geq 0}^k$.

Let $v$ be the canonical generator of $\Delta(\lambda)$. Then 
$v$ is a basis in $\Delta(\lambda)_\lambda$. For $\nu\in\mathfrak{h}^*$,
we have $\Delta(\lambda)_\nu\neq 0$ if and only if $\nu-\lambda$
has the form $\displaystyle\sum_{\alpha\in\mathbf{R}_-}m_\alpha\alpha$,
where all $m_\alpha\in\mathbb{Z}_{\geq 0}$. We can split this sum into
two parts: in the first we sum over all $\alpha\in\mathbf{R}_-$
which are roots of $\mathfrak{a}$ and in the second one, we sum
over all other roots, that is, over $\alpha_1,\alpha_2,\dots,\alpha_k$.
Consider the set $Q$ consisting of all weights $\nu$ such that
$\Delta(\lambda)_\nu\neq 0$ and $\mathbf{p}(\nu)=\eta$.
Then, for  $\zeta\in \mathfrak{h}^*_\mathfrak{a}$,  the inequality
$\mathrm{Res}_\eta(\Delta(\lambda))_\zeta\neq 0$
implies $\zeta\in Q$.

Given a solution $(x_1,x_2,\dots,x_k)\in\mathbb{Z}_{\geq 0}^k$
to Equation~\eqref{eq-gvm}, consider the element
\begin{displaymath}
X_{\alpha_1}^{x_1}X_{\alpha_2}^{x_2}\dots X_{\alpha_k}^{x_k}v
\in\mathrm{Res}_\eta(\Delta(\lambda)).
\end{displaymath}
As we have only finitely many solutions to Equation~\eqref{eq-gvm},
this gives a finite collection of elements. By the PBW Theorem,
this finite collection of elements is a basis of 
$\mathrm{Res}_\eta(\Delta(\lambda))$ 
as a $U(\mathfrak{n}_-)\cap U(\mathfrak{a})$-module.
The claim follows.
\end{proof}

From Lemma~\ref{lem-s2.3-1}, it follows that $\mathrm{Res}_\eta$
defines a functor from $\mathcal{O}(\mathfrak{g})$
to $\mathcal{O}(\mathfrak{a})$, which acts on morphisms by restriction.
Since all modules in $\mathcal{O}(\mathfrak{g})$ are weight modules,
the functor $\mathrm{Res}_\eta$ is exact.

\subsection{Right adjoint of parabolic induction}\label{s2.4}

Given $M\in\mathcal{O}(\mathfrak{g})$
and $\eta\in(\mathfrak{h}^\perp)^*$,
define $\overline{\mathrm{Res}_\eta}(M)$
as the set of all elements $v\in \mathrm{Res}_\eta(M)$
such that $\mathfrak{n}v=0$.

\begin{lemma}\label{lem-s2.4-1}
The set $\overline{\mathrm{Res}_\eta}(M)$ is, in fact,
an $\mathfrak{a}$-submodule of $\mathrm{Res}_\eta(M)$.
\end{lemma}

\begin{proof}
Let $v\in \overline{\mathrm{Res}_\eta}(M)$,
$a\in\mathfrak{a}$ and $n\in \mathfrak{n}$. Then
$na\cdot v=an\cdot v+ [n,a]\cdot v$. Here $an\cdot v=0$
since $n\cdot v=0$ due to $v\in \overline{\mathrm{Res}_\eta}(M)$.
Also, $[n,a]\in \mathfrak{n}$ as the latter is an ideal of
$\mathfrak{p}$ and hence $[n,a]\cdot v=0$ again due to 
$v\in \overline{\mathrm{Res}_\eta}(M)$. Therefore
$na\cdot v=0$, which implies $a\cdot v\in \overline{\mathrm{Res}_\eta}(M)$.
\end{proof}

From Lemmata ~\ref{lem-s2.3-1} and \ref{lem-s2.4-1}, 
it follows that the assignment
$M\mapsto \overline{\mathrm{Res}_\eta}(M)$ is a functor from
$\mathcal{O}(\mathfrak{g})$ to $\mathcal{O}(\mathfrak{a})$.

\begin{proposition}\label{prop-s2.4-2}
The pair $(\mathrm{Ind}_\eta,\overline{\mathrm{Res}_\eta})$ 
is an adjoint pair of functors. 
Furthermore, for $M\in\mathcal{O}(\mathfrak{a})$,
we have $\overline{\mathrm{Res}}_\eta (\mathrm{Ind}_\eta (M))\cong M.$
\end{proposition}

\begin{proof}
The first statement is the usual adjunction between induction and restriction.
The second statement follows from the PBW theorem. 
\end{proof}

As a direct consequence of Proposition~\ref{prop-s2.4-2}, we record that
the functor $\overline{\mathrm{Res}_\eta}$ is left exact.

\subsection{Induced structural modules}\label{s2.5}

Consider now the principal block 
$\mathcal{O}(\mathfrak{g})_0$ in $\mathcal{O}(\mathfrak{g})$.
For $w\in {}_JW$, set $\lambda_w:=w\cdot 0$ and
$\eta_w:=\mathbf{p}(\lambda_w)$. Let $\mu_w\in(\mathfrak{h}_\mathfrak{a})^*$
denote the restriction of $\lambda_w$ to $\mathfrak{h}_\mathfrak{a}$.
Then each $\mu_w$ is a regular, integral and dominant weight for $\mathfrak{a}$.

For a fixed $w\in {}_JW$, the exact parabolic induction
\begin{displaymath}
\mathrm{Ind}_{\eta_w}:\mathcal{O}(\mathfrak{a})_{\mu_w}
\to \mathcal{O}(\mathfrak{g})_{0}
\end{displaymath}
induces a linear map
\begin{displaymath}
[\mathrm{Ind}_{\eta_w}]:
\mathbf{Gr}({}^\mathbb{Z}\mathcal{O}(\mathfrak{a})_{\mu_w})
\to \mathbf{Gr}({}^\mathbb{Z}\mathcal{O}(\mathfrak{g})_{0}).
\end{displaymath}
Under the identification of $\mathbf{Gr}({}^\mathbb{Z}\mathcal{O}(\mathfrak{a})_{\mu_w})$ and $\mathbf{Gr}({}^\mathbb{Z}\mathcal{O}(\mathfrak{g})_{0})$ with the Hecke algebras for $\mathfrak{a}$
and $\mathfrak{g}$, respectively, as in Subsection~\ref{s1.7},
the map $[\mathrm{Ind}_{\eta_w}]$ is given by right multiplication by
$H_w$ in the following way:

\begin{lemma}\label{lem-pindhecke}
For $M\in {}^\mathbb{Z}\mathcal{O}(\mathfrak{a})_{\mu_w}$, 
we have $[\mathrm{Ind}_{\eta_w}(M)]=[M]H_w$.
\end{lemma}

\begin{proof}
By linearity, it is enough to prove this for any collection of modules
whose images in $\mathbf{Gr}({}^\mathbb{Z}\mathcal{O}(\mathfrak{a})_{\mu_w})$
form there a basis. A wise choice of such collection would be standard
modules. For such modules, the claim follows directly from the definitions
and Lemma~\ref{lem-s2.2-2},
combined with the identity $H_uH_w=H_{uw}$, for $u\in W(J)$, since
$\ell(uw)=\ell(u)+\ell(w)$ under our assumption that $w\in{}_JW$.
\end{proof}

Now, for each $u\in W(J)$, we can consider the corresponding induced
structural modules
\begin{displaymath}
\mathrm{Ind}_{\eta_w}(L_\mathfrak{a}(u\cdot \mu_w)),\quad
\mathrm{Ind}_{\eta_w}(\Delta_\mathfrak{a}(u\cdot \mu_w)),\quad
\mathrm{Ind}_{\eta_w}(P_\mathfrak{a}(u\cdot \mu_w)).
\end{displaymath}
As we already established in Lemma~\ref{lem-s2.2-2}, we have
\begin{displaymath}
\Delta_\mathfrak{g}(uw\cdot 0)\cong
\mathrm{Ind}_{\eta_w}(\Delta_\mathfrak{a}(u\cdot \mu_w)).
\end{displaymath}
The induced module $\mathrm{Ind}_{\eta_w}(L_\mathfrak{a}(u\cdot \mu_w))$
is an example of a generalized Verma module. The induced module
$\mathrm{Ind}_{\eta_w}(P_\mathfrak{a}(u\cdot \mu_w))$ will play 
important role in the stratified structure on 
$\mathcal{O}_0$ which we will discuss in the next section.

\begin{proposition}\label{prop-KL-indproj}
For $w\in {}_JW$ and $u\in W(J)$ as above,
the isomorphism in Subsection~\ref{s1.7} sends 
$[\mathrm{Ind}_{\eta_w}(P_\mathfrak{a}(u\cdot \mu_w))]$
to $\underline{H}_uH_w$.
\end{proposition}

\begin{proof}
This follows from the definitions combined with
Lemma~\ref{lem-pindhecke}.
\end{proof}

The elements $\underline{H}_uH_w$ as above form the so-called
{\em hybrid KL basis}, introduced in \cite{GH07}.

{\color{red}Add more references preceding GH07.}

\section{Stratification}\label{s3}

\subsection{Stratified algebras}\label{s3.1}

Let $B$ be a finite dimensional associative algebra. We fix
a complete and irredundant list $L_1,L_2,\dots,L_n$ of representatives
of the isomorphism classes of simple $B$-modules. 
For $1\leq i\leq n$, denote by $P_i$ and $I_i$ the indecomposable 
projective cover and injective envelope of $L_i$, respectively.

Let $\preceq$ be a fixed partial pre-order on the indexing set
$\underline{n}:=\{1,2,\dots,n\}$. For $i,j\in \underline{n}$,
we write $i\sim j$ provided that $i\preceq j$ and $j\preceq i$.
For $i\in \underline{n}$
define 
\begin{itemize}
\item the {\em standard module} $\Delta_i$ as the quotient 
of $P_i$ modulo the trace in $P_i$ of all $P_j$ such that $j\not\preceq i$;
\item the {\em proper standard module} $\overline{\Delta}_i$ as the 
quotient of $\Delta_i$ modulo the trace of all $P_j$ such that 
$j\sim i$ in the radical of $\Delta_i$;
\item the {\em costandard module} $\nabla_i$ as the intersection of
the kernels of all homomorphisms from $I_i$ to $I_j$, for 
$j\not\preceq i$;
\item the {\em proper costandard module} $\overline{\nabla}_i$ as the 
intersection of the kernels of all non-injective homomorphisms from
$\nabla_i$ to $I_j$, for  $j\sim i$.
\end{itemize}
The pair $(B,\preceq)$ is called a {\em stratified algebra} 
provided that all projective $B$-modules have filtrations
with standard subquotients, see \cite{CPS}. Equivalently,
$(B,\preceq)$ is stratified if and only if all injective
modules have filtrations with proper costandard subquotients
see \cite[Theorem~1]{Fr}.
We note that there are many variations of stratified algebras,
see, in particular, \cite{Dl,Fr,BS} and references therein.
Stratified algebras are generalizations of quasi-hereditary algebras,
see \cite{CPS0,DR} for the latter.

\subsection{Parabolic stratification on $\mathcal{O}_0$}\label{s3.2}

Consider the category $\mathcal{O}_0$.
Simple modules in $\mathcal{O}_0$ are bijectively indexed by the 
elements of $W$. Let $\preceq$ be the opposite of the Bruhat
order $\leq$ on $W$. Note that this is a genuine partial order.
For this order, we have $\Delta_w=\overline{\Delta}_w$ is exactly
the Verma module. By \cite{BGG}, all projective modules have 
a filtration with Verma subquotients. Consequently, the algebra
$(A,\preceq)$ is stratified and even quasi-hereditary.

Let us now fix a parabolic subalgebra $\mathfrak{p}$ of $\mathfrak{g}$
as in Section~\ref{s2}. We now want to define a partial pre-order
on $W$ related to $\mathfrak{p}$. Given $x,y\in W$, write
$x=ab$ and $y=a'b'$, where $a,a'\in W(J)$ and $b,b'\in {}_JW$.
Set $x{\preceq^\mathfrak{p}}y$ if and only if $b\preceq b'$.
The subword Property for $W$ ensures that ${\preceq^\mathfrak{p}}$ 
is a refinement of $\preceq$ in the sense that 
$\preceq\subset {\preceq^\mathfrak{p}}$.

Let us start with a description of standard modules for this structure.

\begin{proposition}\label{prop-std}
The standard modules with respect to ${\preceq^\mathfrak{p}}$
are exactly the modules of the form 
$\mathrm{Ind}_{\eta_w}(P_\mathfrak{a}(u\cdot \mu_w))$,
see Subsection~\ref{s2.5}.
\end{proposition}

These standard modules give rise to the 
indecomposable projective objects in the 
Serre subquotients of $\mathcal{O}_0$ discussed in
\cite[Subsection~6.2]{CMZ}.

\begin{proof}
Take $w\in{}_JW$ and $u\in W(J)$.
Our first step is to show that 
$\mathrm{Ind}_{\eta_w}(P_\mathfrak{a}(u\cdot \mu_w))$
has simple top $L_{uw}$. To this end, for $z\in W$, by adjunction,
we have
\begin{displaymath}
\mathrm{Hom}_\mathfrak{g}
(\mathrm{Ind}_{\eta_w}(P_\mathfrak{a}(u\cdot \mu_w)),L_z)=
\mathrm{Hom}_\mathfrak{a}
(P_\mathfrak{a}(u\cdot \mu_w),
\overline{\mathrm{Res}_{\eta_w}}(L_z)).
\end{displaymath}
Assume first that $z\prec^\mathfrak{p}w$.
Let $\tilde{z}$ be the unique element in $W(J)\cap {}_JW$.
Then $z\preceq\tilde{z}\prec w$ by definition. In particular,
$\Delta_z\subset \Delta_{\tilde{z}} \subsetneq \Delta_w$ by 
\cite[Theorem~7.7.7]{Di}. As $z\not \in W(J)w$,
when writing $w\cdot 0-z\cdot 0$ as a linear combination of
positive roots with non-negative coefficients, we must have
a non-zero coefficient at at least one of the positive roots
outside the roots of $\mathfrak{a}$. This implies that any 
weight $\lambda$ of $\Delta_z$ has the same property: 
when writing $w\cdot 0-\lambda$ as a linear combination of
positive roots with non-negative coefficients, we must have
a non-zero coefficient at at least one of the positive roots
outside the roots of $\mathfrak{a}$.
Consequently, ${\mathrm{Res}_{\eta_w}}(L_z)=0$
and therefore $\overline{\mathrm{Res}_{\eta_w}}(L_z)=0$ as well.

If $w{\preceq^\mathfrak{p}}z$, then we claim that
$\overline{\mathrm{Res}_{\eta_w}}(L_z)\neq 0$ implies $z\in W(J)w$.
Indeed, $\overline{\mathrm{Res}_{\eta_w}}(L_z)$, being a
non-zero object in $\mathcal{O}(\mathfrak{a})$, has some simple
submodule, say $L_\mathfrak{a}(\lambda)$, where 
$\lambda\in \mathfrak{h}^*_\mathfrak{a}$. The module
$\mathrm{Ind}_{\eta_w}(L_\mathfrak{a}(\lambda))$ is a quotient of
$\mathrm{Ind}_{\eta_w}(\Delta_\mathfrak{a}(\lambda))$ by exactness, 
and the latter module is 
isomorphic to $\Delta(\mu(\lambda,\eta_w))$ by Lemma~\ref{lem-s2.2-2}.
Hence $\mathrm{Ind}_{\eta_w}(L_\mathfrak{a}(\lambda))$ has simple top
$L(\mu(\lambda,\eta_w))$. By adjunction, it also has a non-zero homomorphism
to $L_z$. Now, the fact that $L_z$ is simple  forces 
$z\cdot 0=\mu(\lambda,\eta_w)$ which implies $z\in W(J)w$.

For $z\in W(J)w$ such that $z=rw$ with $r\in W(J)$, we have 
$\overline{\mathrm{Res}_{\eta_w}}(L_z)=L_\mathfrak{a}(r\cdot \mu_w)$
implying 
\begin{displaymath}
\mathrm{Hom}_\mathfrak{g}
(\mathrm{Ind}_{\eta_w}(P_\mathfrak{a}(u\cdot \mu_w)),L_z)=
\mathrm{Hom}_\mathfrak{a}
(P_\mathfrak{a}(u\cdot \mu_w),
\overline{\mathrm{Res}_{\eta_w}}(L_z))=
\begin{cases}
\mathbb{C},& r=u;\\
0,& r\neq u.
\end{cases}
\end{displaymath}
This proves that $\mathrm{Ind}_{\eta_w}(P_\mathfrak{a}(u\cdot \mu_w))$
has simple top $L_{uw}$ and hence is a quotient of $P_{uw}$.

Next, denote by $K$ be the kernel of of the projection map 
$P_{uw}\tto \mathrm{Ind}_{\eta_w}(P_\mathfrak{a}(u\cdot \mu_w))$.
If $z\in W$ is such that $w\prec^\mathfrak{p}z$, then 
$L_z$ does not appear as a composition factor of 
$\mathrm{Ind}_{\eta_w}(P_\mathfrak{a}(u\cdot \mu_w))$ since
$\mathrm{Ind}_{\eta_w}(P_\mathfrak{a}(u\cdot \mu_w))_{z\cdot 0}=0$.
This means that the trace of $P_z$ in $P_{uw}$ belongs to $K$.
In other words, $K$ contains the trace of all $P_z$,
where $w\prec^\mathfrak{p}z$. In fact, we claim that $K$
coincides with this trace.

To prove the claim, let us compare standard filtrations of 
$P_{uw}$ and $\mathrm{Ind}_{\eta_w}(P_\mathfrak{a}(u\cdot \mu_w))$.
All standard modules which appear as subquotients of a standard
filtration of $P_{uw}$ are of the form $\Delta_z$, where $z\leq uw$,
that is $w{\preceq^\mathfrak{p}}z$.  For $w\prec^\mathfrak{p}z$,
all the corresponding standard modules are in $K$ by the previous
paragraph. Therefore it remains to show that, for 
$z\in W(J)w$ with $z=rw$, where $r\in W(J)$, we have
\begin{equation}\label{eq-comparison}
[P_{uw}:\Delta_z]=
[\mathrm{Ind}_{\eta_w}(P_\mathfrak{a}(u\cdot \mu_w)):\Delta_z]. 
\end{equation}
Note that $\mathrm{Ind}_{\eta_w}(P_\mathfrak{a}(u\cdot \mu_w))$
does have a standard filtration by combining exactness of 
induction and Lemma~\ref{lem-s2.2-2} with the fact that 
$P_\mathfrak{a}(u\cdot \mu_w)$ has a standard filtration.

By the BGG reciprocity, $[P_{uw}:\Delta_z]=[\Delta_z:L_{uw}]$.
At the same time, by the BGG reciprocity combined with
exactness and Lemma~\ref{lem-s2.2-2}, we have
\begin{displaymath}
[\mathrm{Ind}_{\eta_w}(P_\mathfrak{a}(u\cdot \mu_w)):\Delta_z]=
[\Delta_\mathfrak{a}(r\cdot \mu_w):L_\mathfrak{a}(u\cdot \mu_w)].
\end{displaymath}
Finally, { we claim that we have
\begin{equation}\label{eq:form751}
[\Delta_z:L_{uw}]= 
[\Delta_\mathfrak{a}(r\cdot {\mu}_w):
L_\mathfrak{a}(u\cdot {\mu}_w)]. 
\end{equation}
Indeed, we have 
$\Delta_z=\mathrm{Ind}_{\eta_w}(\Delta_\mathfrak{a}(r\cdot {\mu}_w))$
by Lemma~\ref{lem-s2.2-2}.
By PBW Theorem, 
the module $\mathrm{Ind}_{\eta_w}(L_\mathfrak{a}(u\cdot {\mu}_w))$
has $L_{uw}$ as simple top and every other simple subquotient $L$ of
$\mathrm{Ind}_{\eta_w}(L_\mathfrak{a}(u\cdot {\mu}_w))$
satisfies $\overline{\mathrm{Res}_{\eta_w}}(L)=0$.
Hence \eqref{eq:form751} follows from exactness of parabolic induction. 
}
This proves \eqref{eq-comparison} and completes
the proof of our proposition.
\end{proof}

\begin{corollary}\label{cor-pr-std}
The proper standard modules with respect to ${\preceq^\mathfrak{p}}$
are exactly the modules of the form 
$\mathrm{Ind}_{\eta_w}(L_\mathfrak{a}(u\cdot \mu_w))$.
\end{corollary}

\begin{proof}
Recall that the proper standard quotient of 
$\mathrm{Ind}_{\eta_w}(P_\mathfrak{a}(u\cdot \mu_w))$
is given by factoring out the trace $Q$ of all 
$P_a$, where $a\in W(J)w$, in the radical
of $\mathrm{Ind}_{\eta_w}(P_\mathfrak{a}(u\cdot \mu_w))$.

Let $K$ denote the kernel of the natural projection from the 
module
$\mathrm{Ind}_{\eta_w}(P_\mathfrak{a}(u\cdot \mu_w))$
to $\mathrm{Ind}_{\eta_w}(L_\mathfrak{a}(u\cdot \mu_w))$.
Clearly, $Q\subset K$ as 
\begin{displaymath}
\overline{\mathrm{Res}}_{\eta_w}
\big(\mathrm{Ind}_{\eta_w}(L_\mathfrak{a}(u\cdot \mu_w))\big)
\cong L_\mathfrak{a}(u\cdot \mu_w).
\end{displaymath}
So, to complete the proof, we only need to show that 
$K\subset Q$.

{
From the definition, we see that the module $K$ is precisely 
$\mathrm{Ind}_{\eta_w}(\mathrm{rad}(P_{\mathfrak a} (u\cdot\mu_w)))$.
By exactness of induction applied to any composition series of
$\mathrm{rad}(P_{\mathfrak a} (u\cdot\mu_w))$, the module
$K$ has a filtration  whose subquotients have the form
$\mathrm{Ind}_{\eta_w}(L_\mathfrak{a}(z\cdot \mu_w))$,
where $z\in W(J)$. As each such 
$\mathrm{Ind}_{\eta_w}(L_\mathfrak{a}(z\cdot \mu_w))$
is a quotient of $P_{zw}$, we obtain that $K\subset Q$.
}
\end{proof}

By applying the simple preserving duality $\star$, we obtain:

\begin{corollary}\label{cor-costd}
We have:

\begin{enumerate}[$($a$)$]
\item\label{cor-costd.1} The costandard modules with respect 
to ${\preceq^\mathfrak{p}}$ are exactly the modules of the form 
$\mathrm{Ind}_{\eta_w}(P_\mathfrak{a}(u\cdot \mu_w))^\star$.
\item\label{cor-costd.2} The proper costandard modules with respect 
to ${\preceq^\mathfrak{p}}$ are exactly the modules of the form 
$\mathrm{Ind}_{\eta_w}(L_\mathfrak{a}(u\cdot \mu_w))^\star$.
\end{enumerate}
\end{corollary}

We can now formulate and prove the main result of this section.

\begin{theorem}\label{thm-main-1}
The pair $(A,{\preceq^\mathfrak{p}})$ is a stratified algebra.
\end{theorem}

\begin{proof}
All projectives in $\mathcal{O}(\mathfrak{g})$ have a Verma flag.
By Lemma~\ref{lem-s2.2-2}, each Verma module 
$\Delta_z$ is of the form 
$\mathrm{Ind}_{\eta_w}(\Delta_\mathfrak{a}(u\cdot \mu_w))$,
for some $u\in W(J)$ and $w\in {}_JW$. Each 
$\Delta_\mathfrak{a}(u\cdot \mu_w)$ has a composition series. 
By the exactness of parabolic induction and our description of 
proper standard modules in Corollary~\ref{cor-pr-std}, it follows
that $\mathrm{Ind}_{\eta_w}(\Delta_\mathfrak{a}(u\cdot \mu_w))$ 
has a filtration with 
proper standard subquotients. Consequently, all projective modules
in $\mathcal{O}(\mathfrak{g})$ have a proper standard filtration.

From \cite[Theorem~1]{Fr}, it follows that all injective modules in 
$\mathcal{O}(\mathfrak{g})$ have a costandard filtration.

Using $\star$ and our description of 
costandard modules in Corollary~\ref{cor-costd}, it follows
that all projective modules
in $\mathcal{O}(\mathfrak{g})$ have a standard filtration.
\end{proof}

We note that Theorem~\ref{thm-main-1} is inspired by and is a
generalization of the main result of \cite{FKM}.

\begin{example}
{\rm Consider $\mathfrak{g}=\mathfrak{sl}_3$
with $W=\{e,s,t,st,ts,w_0=sts=tst\}$ and choose
$W(J)=\{e,s\}$. Then ${}_JW=\{e,t,ts\}$. We have
\begin{displaymath}
\mathrm{Ind}_{\eta_e}(P_\mathfrak{a}(e\cdot \mu_e))=P_e=\Delta_e,\,\, 
\mathrm{Ind}_{\eta_t}(P_\mathfrak{a}(e\cdot \mu_t))=\Delta_t,\, \,
\mathrm{Ind}_{\eta_{ts}}(P_\mathfrak{a}(e\cdot \mu_{ts}))=\Delta_{ts}.
\end{displaymath}
Further,  $\mathrm{Ind}_{\eta_e}(P_\mathfrak{a}(s\cdot \mu_e))=P_s$
has $\Delta_e$ as a submodule and $\Delta_s$ as quotient.
The module $\mathrm{Ind}_{\eta_t}(P_\mathfrak{a}(s\cdot \mu_t))$
has $\Delta_t$ as a submodule and $\Delta_{st}$ as the corresponding quotient.
This module can be realized as the cokernel of the embedding of
$P_s$ into $P_{st}$. Finally, the module 
$\mathrm{Ind}_{\eta_{ts}}(P_\mathfrak{a}(s\cdot \mu_{ts}))$
has $\Delta_{ts}$ as a submodule and $\Delta_{w_0}$ as the corresponding
quotient. This module can be realized as the cokernel of the embedding of
$P_{st}$ into $P_{w_0}$.

Now about the filtrations of indecomposable projectives. 
As both $P_e$ and $P_s$ are induced modules, for them the filtrations
are obvious. As mentioned in the previous paragraph, 
$P_{st}$ surjects onto $\mathrm{Ind}_{\eta_t}(P_\mathfrak{a}(s\cdot \mu_t))$
with kernel $P_s$. Also, 
$P_{w_0}$ surjects onto 
$\mathrm{Ind}_{\eta_{ts}}(P_\mathfrak{a}(s\cdot \mu_{ts}))$
with kernel $P_{st}$. This gives our filtrations for
$P_{st}$ and $P_{w_0}$.

The module $P_t$ surjects onto $\Delta_t$ with kernel
$\Delta_e$. As both  $\Delta_t$ and $\Delta_e$ are parabolically induced
(the former is $\mathrm{Ind}_{\eta_t}(P_\mathfrak{a}(e\cdot \mu_t))$
while the latter is $\mathrm{Ind}_{\eta_e}(P_\mathfrak{a}(e\cdot \mu_e))$),
we have our filtration for $P_t$.

Finally, the module $P_{ts}$ has $P_s$ as submodule and the cokernel
surjects onto $\Delta_{ts}$ with kernel
$\Delta_t$. As both  $\Delta_{ts}$ and $\Delta_t$ are parabolically induced
(the former is $\mathrm{Ind}_{\eta_{ts}}(P_\mathfrak{a}(e\cdot \mu_{ts}))$
while the latter is $\mathrm{Ind}_{\eta_t}(P_\mathfrak{a}(e\cdot \mu_t))$),
we obtain our filtration for the module $P_{ts}$.\hfill$\blacksquare$
}
\end{example}

\subsection{Consequences}\label{s3.3}

\begin{corollary}\label{cor-s3.3-1}
{\hspace{2mm}}

{

\begin{enumerate}[$($a$)$]
\item\label{cor-s3.3-1.1} The images of the modules
$\mathrm{Ind}_{\eta_w}(P_\mathfrak{a}(u\cdot \mu_w))$, 
where $u\in W(J)$ and $w\in {}_JW$, 
in the Grothendieck group of $\mathcal{O}_0$
form there a basis.
\item\label{cor-s3.3-1.2} The images of the modules
$\mathrm{Ind}_{\eta_w}(P_\mathfrak{a}(u\cdot \mu_w))$, 
where $u\in W(J)$ and $w\in {}_JW$, 
in the Grothendieck group of  ${}^{\mathbb{Z}}\mathcal{O}_0$
form there an $\mathbb{A}$-basis.
\end{enumerate}
}
\end{corollary}

\begin{proof}
Due to our stratification, the transformation matrix from 
the basis of projectives to this new basis is upper triangular
with $1$'s on the diagonal. This implies the {
ungraded claim and the graded claim follows from the ungraded one.}
\end{proof}

The above corollary together with Proposition~\ref{prop-KL-indproj} 
is a categorical lift of the fact that the elements $\underline{H}_uH_w$ 
form a basis in $\mathbf{H}$. 

\begin{corollary}\label{cor-s3.3-2}
When expressed in this new basis, the coefficients 
of any $[P]$, where $P\in \mathcal{O}_0$ is projective, 
are given by (polynomials whose coefficients are)
non-negative integers.
\end{corollary}

\begin{proof}
This is a direct consequence of the fact that all
projectives have a filtration by standard modules.
\end{proof}

In particular, the above corollary proves a categorical 
analogue of the fact that KL basis expands non-negatively 
with respect to a hybrid basis, see~\cite[Theorem~3.2]{GH07}.

\begin{corollary}\label{cor-s3.3-5}
Let $\mathfrak{q}\subset\mathfrak{p}\subset\mathfrak{g}$
be two parabolic subalgebras. Then all standard modules with respect to
$\preceq^{\mathfrak{p}}$ have a filtration with subquotients being
standard modules with respect to $\preceq^{\mathfrak{q}}$.
\end{corollary}

\begin{proof}
For $w\in {}_JW$, consider { the direct sum of 
all $\mathrm{Ind}_{\eta_w}(P_\mathfrak{a}(u\cdot \mu_w))$, where
$u\in W(J)$. Using either the equivalence in \cite[Theorem~32]{CMZ}
or Proposition~\ref{prop-s2.4-2}, we see that
the endomorphism algebra of this direct sum is isomorphic to the
endomorphism algebra of a basic projective generator for
$\mathcal{O}_0(\mathfrak{a})$, that is, the direct sum of all
$P_\mathfrak{a}(u\cdot \mu_w)$, over all $u\in W(J)$.}

Note that 
$\mathfrak{a}\cap\mathfrak{q}$ is a parabolic subalgebra
of $\mathfrak{a}$. Now the claim follows from the 
exactness of the parabolic induction, applying 
Theorem~\ref{thm-main-1} to the algebra $\mathfrak{a}$
and its parabolic subalgebra $\mathfrak{a}\cap\mathfrak{q}$.
\end{proof}

The above statement give a categorical interpretation of 
a results in \cite{BMS}.


\subsection{Action of projective functors}\label{s3.4}

In this section we describe an interesting (and expected) property
of modules having (proper) (co)-standard filtrations.

\begin{proposition}\label{prop-s3.4-1}
For $\mathfrak{p}\subset \mathfrak{g}$ as above, the full 
subcategory of all modules in $\mathcal{O}_0(\mathfrak{g})$
that admit a filtration with proper standard subquotients is
stable under the action of projective functors.
\end{proposition}

\begin{proof}
Denote by $\mathfrak{n}_-^\mathfrak{p}$ the linear span of
all $\mathfrak{g}_\alpha$, where $\alpha\in\mathbf{R}_-$
and $\mathfrak{g}_\alpha\not\subset \mathfrak{a}$. We claim that
the full  subcategory $\mathcal{X}$
of all modules in $\mathcal{O}_0(\mathfrak{g})$
that admit a filtration with proper standard subquotients
coincides with the full  subcategory $\mathcal{Y}$ of all modules in 
$\mathcal{O}_0(\mathfrak{g})$ which are free over 
$U(\mathfrak{n}_-^\mathfrak{p})$ (possibly of infinite rank).

As each proper standard module is obtained by parabolic induction,
it is free over $U(\mathfrak{n}_-^\mathfrak{p})$ by construction.
This implies $\mathcal{X}\subset \mathcal{Y}$.

To prove the opposite inclusions, we use the argument similar
to \cite[Proposition~2]{BGG} and proceed by induction on 
the length of a module. Let $M\in \mathcal{Y}$.
If $M=0$, then $M\in \mathcal{X}$ and we are done.
This is the basis of our induction.

To prove the induction step, assume $M\neq 0$.
Choose $w\in {}_JW$ maximal with respect to ${\preceq^\mathfrak{p}}$
such that $\mathrm{Res}_{\eta_w}(M)\neq 0$. 
{The maximality of $w$, in particular, implies that 
$\mathfrak{n}\mathrm{Res}_{\eta_w}(M)=0$, equivalently, we have
$\mathrm{Res}_{\eta_w}(M)=\overline{\mathrm{Res}_{\eta_w}}(M)$.}
By exactness of induction, any composition series of
$\mathrm{Res}_{\eta_w}(M)$ gives rise a filtration of 
$\mathrm{Ind}_{\eta_w}\mathrm{Res}_{\eta_w}(M)$ with proper standard
subquotients. Therefore 
$\mathrm{Ind}_{\eta_w}\mathrm{Res}_{\eta_w}(M)\in \mathcal{X}$.

Since the action of $U(\mathfrak{n}_-^\mathfrak{p})$ on $M$ is free,
by adjunction, the module $\mathrm{Ind}_{\eta_w}\mathrm{Res}_{\eta_w}(M)$
injects into $M$. Due to the maximality assumption on $w$, the action of 
$U(\mathfrak{n}_-^\mathfrak{p})$ on the quotient module
$M/\mathrm{Ind}_{\eta_w}\mathrm{Res}_{\eta_w}(M)$ remains 
free. 
Moreover, this quotient has strictly smaller length than $M$.
This completes our induction step and hence shows that 
$\mathcal{Y}\subset \mathcal{X}$, implying $\mathcal{X}=\mathcal{Y}$.

The category $\mathcal{Y}$ (and hence the category $\mathcal{X}$ as well) 
is obviously closed with respect to 
tensoring with finite dimensional $\mathfrak{g}$-modules. 
Finally, the category $\mathcal{X}$, being homologically dual to 
the category of modules with costandard filtration, is idempotent
split. As any projective functor is a direct summand of tensoring
with some finite dimensional $\mathfrak{g}$-module, 
the claim of the proposition follows.
\end{proof}

\begin{corollary}\label{cor-s3.4-2}
For $\mathfrak{p}\subset \mathfrak{g}$ as above, we have:
\begin{enumerate}[$($a$)$]
\item \label{cor-s3.4-2.1} The full 
subcategory of all modules in $\mathcal{O}_0(\mathfrak{g})$
that admit a filtration with proper costandard subquotients is
stable under the action of projective functors.
\item \label{cor-s3.4-2.2} The full 
subcategory of all modules in $\mathcal{O}_0(\mathfrak{g})$
that admit a filtration with standard subquotients is
stable under the action of projective functors.
\item \label{cor-s3.4-2.3} The full 
subcategory of all modules in $\mathcal{O}_0(\mathfrak{g})$
that admit a filtration with costandard subquotients is
stable under the action of projective functors.
\end{enumerate}
\end{corollary}

\begin{proof}
Claim~\eqref{cor-s3.4-2.1} follows from 
Proposition~\ref{prop-s3.4-1} by duality since the
latter swaps proper standard and proper costandard modules
and commutes with the action of projective functors.
Similarly, Claim~\eqref{cor-s3.4-2.2}
follows from Claim~\eqref{cor-s3.4-2.3} by duality.

Claim~\eqref{cor-s3.4-2.3} follows from Proposition~\ref{prop-s3.4-1}
by adjunction as the category of modules with {
proper standard filtration is left homologically dual to the category of
modules with costandard filtration.}
\end{proof}

{
The graded lift of 
Corollary~\ref{cor-s3.4-2}\eqref{cor-s3.4-2.2} gives a conceptual
categorical explanation of \cite[Theorem~3.2]{GH07}.}

\subsection{Tilting modules}\label{s3.6}

Another standard aspect of stratified algebras is the titling
theory, originated in \cite{Ri} and further developed
in \cite{AHLU,Fr,FM}. A tilting module with respect to our
stratified structure is a module which admits both a 
standard and a proper {costandard} filtration. From the previous subsection,
we know that both, the category of modules which admit a standard
filtration and the category of modules which admit a proper
{costandard} filtration, are closed with respect to the action of 
projective functors. Consequently, the category of tilting modules
is closed under the action of projective functors. This 
allows us to describe indecomposable tilting modules for our
stratified structure in a fairly easy way.
Recall that $w_0^\mathfrak{p}$ denotes the longest element
in $W(J)$.

\begin{proposition}\label{prop-tilting1}
The indecomposable tilting modules with respect to our
stratified structure are exactly the modules 
$\theta_x\Delta_{w_0^\mathfrak{p}w_0}$,
where $x\in W$.
\end{proposition}

\begin{proof}
Let $w:=w_0^\mathfrak{p}w_0\in {}_JW$. Hence
$\Delta_{w}=\mathrm{Ind}_{\eta_w}(P_\mathfrak{a}(e\cdot \mu_w))$
is a standard module by Lemma~\ref{lem-s2.2-2}.
Also note that  $w$ is a minimal element with respect to 
$\preceq^\mathfrak{p}$. Hence $\Delta_{w}$ is a tilting module.
Indeed, it is a standard module by definition and it has
a filtration by proper costandard modules since the latter are simple
due to minimality of $w$.
Applying $\theta_x$, for $x\in W$, produces tilting modules,
as pointed out above. Moreover, all $\theta_x\Delta_w$
are indecomposable as the outcome of an application of an indecomposable
projective functor to a Verma module is always indecomposable.
Finally, all $\theta_x\Delta_w$ are pairwise non-isomorphic.
This follows by combining \cite[Corollary~6.4]{Jo} with 
\cite[Proposition~7.2]{KMM2}.
This gives the correct number $|W|$ of pairwise non-isomorphic
indecomposable tilting modules.
The claim of the lemma follows.
\end{proof}

For $w\in W$, let $\top_w$ denote the twisting functor for $w$,
see \cite{AS}. We also denote by 
$\mathbf{G}_w:=\star\circ\top_{w^{-1}}\circ\star$
the right adjoint of $\top_w$, usually called a {\em completion functor},
see \cite{MS2}. Recall that twisting functors
(and hence also completion functors) commute with projective 
functors. Moreover, 
$\mathbf{G}_{w_0^\mathfrak{p}}\Delta_{w_0}=\Delta_{w_0^\mathfrak{p}w_0}$.
Also, all classical tilting modules in $\mathcal{O}_0$ are
exactly the modules of the form $\theta_x\Delta_{w_0}$,
where $x\in W$. Consequently, the indecomposable tilting modules 
with respect to our stratified structure are exactly the 
$\mathbf{G}_{w_0^\mathfrak{p}}$-twists of the usual tilting modules
in $\mathcal{O}_0$ (we also note that twisting functors are
exact on modules with Verma flag, see \cite{AS}, hence completion
functors are exact on modules with dual Verma flag).

\subsection{The zoo of different bases in the Grothendieck group}\label{s3.5}

As already mentioned in Subsections~\ref{s1.5} and \ref{s1.7}, we
have the following natural bases in 
$\mathbf{H}=\mathbf{Gr}({}^\mathbb{Z}\mathcal{O}_0)$:

{\bf The standard basis} $\{H_w\,:\,w\in W\}$. This corresponds
to the classes of Verma modules $\{[\Delta_w]\,:\,w\in W\}$.

{\bf The Kazhdan Lusztig basis} $\{\underline{H}_w\,:\,w\in W\}$. 
This corresponds to the classes of indecomposable projective 
modules $\{[P_w]\,:\,w\in W\}$. All entries of the transformation 
matrix from the KL basis to the standard basis are polynomials
with non-negative coefficients.

{\bf The dual Kazhdan-Lusztig basis} $\{\underline{\hat{H}}_w\,:\,w\in W\}$. 
This corresponds to the classes of simple modules $\{[L_w]\,:\,w\in W\}$.
Note that simple modules are right homologically dual to projective modules,
that is, for all $i,j\in\mathbb{Z}$ and all $u,w\in W$, we have:
\begin{displaymath}
\mathrm{Ext}^{i}(P_u,L_w\langle j\rangle) =
\begin{cases}
\mathbb{C},& i=j=0\text{ and }u=w;\\
0,& \text{else}. 
\end{cases}
\end{displaymath}
All entries of the transformation  matrix from the standard basis 
to the dual KL basis are polynomials
with non-negative coefficients.

Of course, we also have:

{\bf The costandard basis} which corresponds
to the classes of dual Verma modules, that is, $\{[\nabla_w]\,:\,w\in W\}$.
Note that dual Verma modules are right homologically dual to Verma modules,
that is, for all $i,j\in\mathbb{Z}$ and all $u,w\in W$, we have:
\begin{displaymath}
\mathrm{Ext}^{i}(\Delta_u,\nabla_w\langle j\rangle) =
\begin{cases}
\mathbb{C},& i=j=0\text{ and }u=w;\\
0,& \text{else}. 
\end{cases}
\end{displaymath}

{\bf The injective Kazhdan-Lusztig  basis} which corresponds
to the classes of indecomposable injective modules $\{[I_w]\,:\,w\in W\}$.
Note that injective modules are right homologically dual to simple modules,
that is, for all $i,j\in\mathbb{Z}$ and all $u,w\in W$, we have:
\begin{displaymath}
\mathrm{Ext}^{i}(L_u,I_w\langle j\rangle) =
\begin{cases}
\mathbb{C},& i=j=0\text{ and }u=w;\\
0,& \text{else}. 
\end{cases}
\end{displaymath}

{\bf The twisted Kazhdan-Lusztig  basis} which corresponds
to the classes of indecomposable tilting modules $\{[T_w]\,:\,w\in W\}$.
As all tilting modules have Verma flags, all entries of the transformation 
matrix from the twisted KL basis to the standard basis are polynomials
with non-negative coefficients.

Now, given $J\subset \pi$, Theorem~\ref{thm-main-1} implies existence
of the following additional bases:

{\bf The hybrid basis} $\{\underline{H}_uH_w\,:\,u\in W(J),\,w\in {}_JW\}$
which corresponds to the classes of standard modules in our
stratification, that is, the following classes:
\begin{displaymath}
\{[\mathrm{Ind}_{\eta_w}
(P_\mathfrak{a}(u\cdot \mu_w))]\,:\, u\in W(J),\,w\in {}_JW\}.
\end{displaymath}
As all projective modules have a standard filtration,
all entries of the transformation  matrix from the KL basis to 
the hybrid basis are polynomials with non-negative coefficients.

{\bf The proper hybrid basis} 
$\{\underline{\hat{H}}_uH_w\,:\,u\in W(J),\,w\in {}_JW\}$
which corresponds to the classes of proper standard modules in our
stratification, that is, the following classes:
$\{[\mathrm{Ind}_{\eta_w}
(L_\mathfrak{a}(u\cdot \mu_w))]\,:\, u\in W(J),\,w\in {}_JW\}$.
As all standard modules have a proper standard filtration,
all entries of the transformation  matrix from the hybrid basis to 
the proper hybrid basis are polynomials with non-negative coefficients.

{\bf The dual hybrid basis} which corresponds to the classes
of costandard modules in our stratification, that is, the
classes $\{[\mathrm{Ind}_{\eta_w}
(P_\mathfrak{a}(u\cdot \mu_w))^\star]\,:\, u\in W(J),\,w\in {}_JW\}$.
Note that the costandard modules are right homologically dual to 
proper standard modules,
that is, for all $i,j\in\mathbb{Z}$ and all $u_1,u_2\in W(J)$
and $w_1,w_2\in {}_JW$, we have:

\resizebox{\textwidth}{!}{
$
\mathrm{Ext}^{i}(\mathrm{Ind}_{\eta_{w_1}}
(L_\mathfrak{a}(u_1\cdot \mu_{w_1})),
\mathrm{Ind}_{\eta_{w_2}}
(P_\mathfrak{a}(u_2\cdot \mu_{w_2}))^\star\langle j\rangle) =
\begin{cases}
\mathbb{C},& i=j=0,\,u_1=u_2,\,w_1=w_2;\\
0,& \text{else}. 
\end{cases}
$
}

{\bf The dual proper hybrid basis} which corresponds to the classes
of proper costandard modules in our stratification, that is, 
the classes 
\begin{displaymath}
\{[\mathrm{Ind}_{\eta_w}
(L_\mathfrak{a}(u\cdot \mu_w))^\star]\,:\, u\in W(J),\,w\in {}_JW\}. 
\end{displaymath}
Note that the proper costandard modules are right homologically dual to 
standard modules, that is, for all $i,j\in\mathbb{Z}$ and all 
$u_1,u_2\in W(J)$ and $w_1,w_2\in {}_JW$, we have:

\resizebox{\textwidth}{!}{
$
\mathrm{Ext}^{i}(\mathrm{Ind}_{\eta_{w_1}}
(P_\mathfrak{a}(u_1\cdot \mu_{w_1})),
\mathrm{Ind}_{\eta_{w_2}}
(L_\mathfrak{a}(u_2\cdot \mu_{w_2}))^\star\langle j\rangle) =
\begin{cases}
\mathbb{C},& i=j=0,\,u_1=u_2,\,w_1=w_2;\\
0,& \text{else}. 
\end{cases}
$
}

Finally, we note that there are also the other side versions of
all the latter bases as will be discussed in the next section.

\begin{example}\label{sl3-ebases}
{\rm
Consider $\mathcal{O}_{0}(\mathfrak{sl}_3)$. Then
$W=\{e,s,t,st,ts,w_0=sts=tst\}$ and we have 
\begin{gather*}
\underline{H}_e=H_e,\quad  
\underline{H}_s=H_s+vH_e,\quad  
\underline{H}_t=H_t+vH_e,\\  
\underline{H}_{st}=H_{st}+vH_s+vH_t+v^2H_e,\quad  
\underline{H}_{ts}=H_{t}+vH_s+vH_t+v^2H_e,\\  
\underline{H}_{w_0}=H_{w_0}+vH_{st}+vH_{ts}+v^2H_s+v^2H_t+v^3H_e.  
\end{gather*}
So, the transformation matrix between the standard and the KL bases is:
\begin{displaymath}
T_1:=\left(
\begin{array}{cccccc}
1&v&v&v^2&v^2&v^3\\ 
0&1&0&v&v&v^2\\ 
0&0&1&v&v&v^2\\ 
0&0&0&1&0&v\\ 
0&0&0&0&1&v\\ 
0&0&0&0&0&1
\end{array}
\right) 
\end{displaymath}
Let $W(J)=\{e,s\}$, then ${}_JW=\{e,t,ts\}$ and hence the corresponding
hybrid basis consists of the following elements:
\begin{gather*}
\underline{H}_eH_e=H_e,\quad
\underline{H}_sH_e=H_s+vH_e,\quad
\underline{H}_eH_t=H_t, \\
\underline{H}_sH_t=H_{st}+vH_t, \quad
\underline{H}_eH_{ts}=H_{ts}, \quad
\underline{H}_sH_{ts}=H_{w_0}+vH_{ts}. 
\end{gather*}
the transformation matrices between the standard and the
hybrid basis and between the hybrid basis and the 
KL basis, respectively, are:
\begin{displaymath}
T_2:=\left(
\begin{array}{cccccc}
1&v&0&0&0&0\\ 
0&1&0&0&0&0\\ 
0&0&1&v&0&0\\ 
0&0&0&1&0&0\\ 
0&0&0&0&1&v\\ 
0&0&0&0&0&1
\end{array}
\right) 
\quad\text{ and }\quad
T_3:= \left(\begin{array}{cccccc}
1&0&v&0&0&0\\ 
0&1&0&v&v&v^2\\ 
0&0&1&0&v&0\\ 
0&0&0&1&0&v\\ 
0&0&0&0&1&0\\ 
0&0&0&0&0&1
\end{array}
\right) 
\end{displaymath}
It is easy to check that $T_1=T_2T_3$.
\hfill$\blacksquare$
}
\end{example}

\section{The other side}\label{s4}

\subsection{Motivation}\label{s4.1}

There is an apparent asymmetry in the elements 
$\underline{H}_uH_w$ discussed in Proposition~\ref{prop-KL-indproj}.
Here we see that an element of the KL basis is on the left while
an element of the standard basis is on the right. Of course, 
the sides could be swapped as follows: For $J\subset\pi$,
consider the elements of the form 
$H_w\underline{H}_u$, where $w\in W_J$ and $u\in W(J)$.
These form the other version of the {\em hybrid KL basis}.

Moreover, since the action of projective functors on 
$\mathcal{O}_0$ is a right action, the elements
$H_w\underline{H}_u$ have a very clear counterparts in
$\mathcal{O}_0$, namely, $[\theta_u \Delta_w]= H_w\underline{H}_u$.
The elements $\theta_u \Delta_w$ give rise to the 
indecomposable projective objects in the 
Serre subquotients of $\mathcal{O}_0$ discussed in
\cite[Subsection~6.3]{CMZ}.

This naturally leads to the following question: do projectives
in $\mathcal{O}_0$ have a filtration with subquotients 
isomorphic to such $\theta_u \Delta_w$? We will answer this
question positively in this section. Our approach would be to
reduce the question to the one we already solved in the 
previous section, that is, for the other side. This reduction
is by no means trivial as our main tool, namely, parabolic induction,
is no longer available. The underlying observation for 
this non-triviality is the combination of the fact that 
the KL combinatorics is invariant under the involution 
$x\mapsto x^{-1}$ on $W$ with the observation in
\cite[Remark~1.2]{MS2} that this is not categorifiable, that is,
there is no functorial involutive equivalence on $\mathcal{O}_0$ 
which maps  $L_x$ to $L_{x^{-1}}$.

\subsection{The result}\label{s4.2}

We fix a parabolic subalgebra $\mathfrak{p}$ of $\mathfrak{g}$
as in Section~\ref{s2}. Given $x,y\in W$, write
$x=ab$ and $y=a'b'$, where $a,a'\in W_J$ and $b,b'\in W(J)$.
Set $x{\preceq_\mathfrak{p}}y$ if and only if $a\preceq a'$.
The Subword Property for $W$ ensures that ${\preceq_\mathfrak{p}}$ 
is a refinement of $\preceq$ in the sense that 
$\preceq\subset {\preceq_\mathfrak{p}}$. Our second main result
is the following: 

\begin{theorem}\label{thm-main2}
We have: 

\begin{enumerate}[$($a$)$]
\item\label{thm-main2.1} The standard modules with respect 
to ${\preceq_\mathfrak{p}}$
are exactly the modules of the form 
$\theta_u \Delta_w$ as in Subsection~\ref{s4.1}.
\item\label{thm-main2.2}
The pair $(A,{\preceq_\mathfrak{p}})$ is a stratified algebra.
\end{enumerate}
\end{theorem}

\subsection{Harish-Chandra bimodules}\label{s4.3}

In order to prove Theorem~\ref{thm-main2}, we need to recall
one more technical tool, namely, that of Harish-Chandra 
$\mathfrak{g}$-$\mathfrak{g}$-bimodules. We refer to 
\cite{BG,Ja} for details.

A $\mathfrak{g}$-$\mathfrak{g}$-bimodule $X$ is called a
{\em Harish-Chandra bimodule} if it is finitely generated
and the adjoint action of $\mathfrak{g}$ on it is locally finite
and has finite multiplicities. 

Let $Z(\mathfrak{g})$ be the center of $U(\mathfrak{g})$
and $\mathtt{m}$ be the maximal ideal in $Z(\mathfrak{g})$
given by the kernel of the action of $Z(\mathfrak{g})$ on 
the trivial $\mathfrak{g}$-module. We denote by $\mathcal{H}$
the category of all Harish-Chandra $\mathfrak{g}$-$\mathfrak{g}$-bimodules.
As usual, we consider the full subcategory 
${}^{\infty}_{\,\,0}\mathcal{H}^1_0$ of 
$\mathcal{H}$ which consists of all bimodules the right action of
$\mathtt{m}$ on which is zero and the left action of 
$\mathtt{m}$ on which is locally nilpotent.
We also consider the full subcategory 
${}^{1}_{0}\mathcal{H}^1_0$ of 
${}^{\infty}_{\,\,0}\mathcal{H}^1_0$ which consists of all 
bimodules the left action of 
$\mathtt{m}$ on which is zero. Note that 
${}^{1}_{0}\mathcal{H}^1_0$ is a full subcategory of 
${}^{\infty}_{\,\,0}\mathcal{H}^1_0$, by definition.

Mapping $X$ to $\displaystyle X\bigotimes_{U(\mathfrak{g})}\Delta_e$
defines an equivalence between ${}^{\infty}_{\,\,0}\mathcal{H}^1_0$
and $\mathcal{O}_0$, see \cite[Theorem~5.9]{BG}. The image of
${}^{1}_{0}\mathcal{H}^1_0$ under this equivalence is the full
subcategory ${\mathcal{O}}_0^\dagger$ of $\mathcal{O}_0$ consisting
of all modules killed by $\mathtt{m}$ (in other words, the action of
the center on which is scalar, that is, such modules have a central character).
Note that ${\mathcal{O}}_0^\dagger$ contains all simple modules
and all Verma modules.

By \cite[Theoreme~1]{So86}, swapping the sides of the bimodule
induces a self-equivalence of ${}^{1}_{0}\mathcal{H}^1_0$ which
translates into a self equivalence $\Phi$ of ${\mathcal{O}}_0^\dagger$
sending $L_x$ to $L_{x^{-1}}$ and $\Delta_x$ to $\Delta_{x^{-1}}$,
for $x\in W$. 

\subsection{Proof of Theorem~\ref{thm-main2}}\label{s4.4}

For $w\in W_J$ and $u\in W(J)$, define 
\begin{displaymath}
Q(w,u):=
\Phi(\mathrm{Ind}_{\eta_{w^{-1}}}(L_\mathfrak{a}(u^{-1}\cdot \mu_{w^{-1}}))).
\end{displaymath}

\begin{lemma}\label{lem-s4.4-1}
The modules $Q(w,u)$ above are the proper standard modules with respect to
$\preceq_\mathfrak{p}$.
\end{lemma}

\begin{proof}
We know from Corollary~\ref{cor-pr-std} that the modules 
$\mathrm{Ind}_{\eta_{w^{-1}}}(L_\mathfrak{a}(u^{-1}\cdot \mu_{w^{-1}}))$ 
are exactly the proper standard modules with respect to the pre-order
$\preceq^\mathfrak{p}$. Each such module
$\mathrm{Ind}_{\eta_{w^{-1}}}(L_\mathfrak{a}(u^{-1}\cdot \mu_{w^{-1}}))$
is a quotient of $\Delta_{(wu)^{-1}}$, it has simple top
$L_{(wu)^{-1}}$ and all other 
composition subquotients are strictly smaller
with respect to $\preceq^\mathfrak{p}$. Moreover,
$\mathrm{Ind}_{\eta_{w^{-1}}}(L_\mathfrak{a}(u^{-1}\cdot \mu_{w^{-1}}))$
is the maximal quotient of $\Delta_{(wu)^{-1}}$ with these
properties.

The mapping $x\mapsto x^{-1}$, which underlines the combinatorics of 
$\Phi$, is an isomorphism between $\preceq^\mathfrak{p}$ with
$\preceq_\mathfrak{p}$. Consequently, $Q(w,u)$ is a quotient of
$\Delta_{wu}$, it has simple top $L_{wu}$ and all other 
composition subquotients of $Q(w,u)$ are strictly smaller
with respect to $\preceq_\mathfrak{p}$. Also, $Q(w,u)$ is
the maximal quotient of $\Delta_{wu}$ with these properties.

The module $P_{wu}$ surjects onto $\Delta_{wu}$ and the kernel has
a filtration by $\Delta_z$ with $z\leq wu$, in particular, we have
$wu\preceq_\mathfrak{p}z$. This implies that the proper costandard
quotient of $P_{wu}$ must be a quotient of $\Delta_{wu}$ and hence
coincides with $Q(w,u)$ by the previous paragraph.
\end{proof}

\begin{lemma}\label{lem-s4.4-2}
All projective modules in $\mathcal{O}_0$ have a filtration
with subquotients isomorphic to the modules $Q(w,u)$ as above.
\end{lemma}

\begin{proof}
By exactness of parabolic induction, we know that all Verma modules
have a filtration with subquotients of the form 
$\mathrm{Ind}_{\eta_{w^{-1}}}(L_\mathfrak{a}(u^{-1}\cdot \mu_{w^{-1}}))$.
Applying $\Phi$, we get that all Verma modules have a filtration 
with subquotients $Q(w,u)$. As all projective in 
$\mathcal{O}_0$ have a filtration with Verma subquotients,
the claim follows.
\end{proof}

Applying $\star$ to the claim of Lemma~\ref{lem-s4.4-2}, we get 
that all injectives in $\mathcal{O}_0$ have a filtration
with subquotients isomorphic to the proper costandard modules 
$Q(w,u)^\star$. Since standard and proper costandard modules
are homologically dual, see \cite[Theorem~1]{Fr}, we obtain
that all projectives in $\mathcal{O}_0$ have a filtration
with standard subquotients. This proves Claim~\eqref{thm-main2.2}
of Theorem~\ref{thm-main2}.

To prove Claim~\eqref{thm-main2.1}
of Theorem~\ref{thm-main2}, we note that, by 
\cite[Theorem~37]{CMZ}, the modules 
$\theta_u\Delta_w$, for $w\in W_J$ and $u\in W(J)$, give rise
to indecomposable projective modules in certain Serre subquotients
of $\mathcal{O}_0$. In particular, each $\theta_u\Delta_w$ has
simple top $L_{wu}$. This means that $\theta_u\Delta_w$
is a quotient of $P_{wu}$. As both $\theta_u\Delta_w$
and $P_{wu}$ have Verma flags, the kernel $K$ of
$P_{wu}\tto \theta_u\Delta_w$ has Verma flag. 

Now let us note that, under the map $x\mapsto x^{-1}$, the
Verma flag combinatorics of $P_{wu}$ and $P_{(wu)^{-1}}$
match (since KL combinatorics is invariant under this map).
This map also matches $H_w\underline{H}_u=[\theta_u\Delta_w]$ with
\begin{displaymath}
\underline{H}_{u^{-1}}h_{w^{-1}}=
[\mathrm{Ind}_{\eta_{w^{-1}}}(P_\mathfrak{a}(u^{-1}\cdot \mu_{w^{-1}}))] 
\end{displaymath}
and hence must match the Verma combinatorics of $K$ with that of 
the kernel $N$ of the projection from $P_{(wu)^{-1}}$ onto
$\mathrm{Ind}_{\eta_{w^{-1}}}(P_\mathfrak{a}(u^{-1}\cdot \mu_{w^{-1}}))$.

But we know from Proposition~\ref{prop-std} that 
$\mathrm{Ind}_{\eta_{w^{-1}}}(P_\mathfrak{a}(u^{-1}\cdot \mu_{w^{-1}}))$
are exactly the standard modules for $\preceq^\mathfrak{p}$. 
In particular, all Verma subquotients appearing in $N$ are 
of the form $\Delta_z$, where $(wu)^{-1}\preceq^\mathfrak{p}z$.
This means that all Verma module appearing in $K$ are 
of the form $\Delta_{z^{-1}}$, where $wu\preceq_\mathfrak{p}z^{-1}$.
Consequently, the standard quotient of $P_{wu}$ must be a quotient
of $\theta_u\Delta_w$. However, since all Verma modules appearing in
$\theta_u\Delta_w$ are of the form $\Delta_{wv}$, where $v\in W(J)$,
we conclude that $\theta_u\Delta_w$ is, in fact, already 
a standard module for $\preceq_\mathfrak{p}$.
This completes the proof of Theorem~\ref{thm-main2}.

\begin{example}
{\rm Consider $\mathfrak{g}=\mathfrak{sl}_3$
with $W=\{e,s,t,st,ts,w_0=sts=tst\}$ and choose
$W(J)=\{e,s\}$. Then $W_J=\{e,t,st\}$. We have
$P_e=\Delta_e$ and $P_s=\theta_s\Delta_e$.

Further, $P_t$ surjects onto $\Delta_t$ with kernel $\Delta_e$.
The module $P_{ts}=\theta_s P_t$ surjects onto
$\theta_s\Delta_t$ with kernel $\theta_s\Delta_e$.
The module $P_{st}$ surjects onto $\Delta_{st}$ with kernel
having a filtration with subquotients $\Delta_t$
and $\theta_s P_t$. Finally, $P_{w_0}$ surjects onto
$\theta_s\Delta_{st}$ with kernel $P_{ts}$.\hfill$\blacksquare$
}
\end{example}

\subsection{Consequences}\label{s4.5}

\begin{corollary}\label{cor-s4.5-1}
The images of the modules
$\theta_u\Delta_w$, 
where $u\in W(J)$ and $w\in W_J$,
in the (graded) Grothendieck group of $\mathcal{O}_0$
form there a basis. 
\end{corollary}

\begin{proof}
Mutatis mutandis the proof of Corollary~\ref{cor-s3.3-1}.
\end{proof}

\begin{corollary}\label{cor-s4.5-2}
When expressed in this new basis, the coefficients 
of any $[P]$, where $P\in \mathcal{O}_0$ is projective, 
are given by (polynomials whose coefficients are)
non-negative integers.
\end{corollary}

\begin{proof}
This is a direct consequence of the fact that all
projectives have a filtration by standard modules.
\end{proof}

\begin{corollary}\label{cor-s4.3-57}
Let $\mathfrak{q}\subset\mathfrak{p}\subset\mathfrak{g}$
be two parabolic subalgebras. Then all standard modules with respect to
$\preceq_{\mathfrak{p}}$ have a filtration with subquotients being
standard modules with respect to $\preceq_{\mathfrak{q}}$.
\end{corollary}

\begin{proof}
For $w\in W_J$, consider the endomorphism algebra of the direct sum of 
all $\theta_u\Delta_w$, where $u\in W(J)$. Then, by the 
equivalence in \cite[Theorem~37]{CMZ}, 
the endomorphism algebra of this direct sum is isomorphic to 
endomorphism algebra of a basic projective generator for
$\mathcal{O}_0(\mathfrak{a})$. Note that 
$\mathfrak{a}\cap\mathfrak{q}$ is a parabolic subalgebra
of $\mathfrak{a}$. Since the equivalence in \cite[Theorem~37]{CMZ}
matches Verma modules, the claim now follows by applying 
Theorem~\ref{thm-main2} to the algebra $\mathfrak{a}$
and its parabolic subalgebra $\mathfrak{a}\cap\mathfrak{q}$.
\end{proof}

The above statements give a categorical interpretation of 
various results from \cite{GH07,BMS}.

\subsection{Tilting modules}\label{s4.6}

For a simple reflection $s$, denote by $\mathbf{C}_s$
and $\mathbf{K}_s$ the corresponding {\em shuffling}
and {\em coshuffling} endofunctors of $\mathcal{O}_0$,
see \cite{MS2}. Note that 
$\mathbf{C}_s\cong\star\circ\mathbf{K}_s\circ\star$.
By \cite[Proposition~7.1]{J-Z}, these functors satisfy
braid relations and hence, for any element $w\in W$, we can define
the corresponding $\mathbf{C}_w$ and $\mathbf{K}_w$,
uniquely up to isomorphism.

Similarly to Proposition~\ref{prop-tilting1},
we have the following description of the indecomposable
tilting modules with respect to our 
stratified structure given by $\preceq_\mathfrak{p}$:

\begin{proposition}\label{prop-tilting2}
The indecomposable tilting modules with respect to our
stratified structure  given by $\preceq_\mathfrak{p}$ 
are exactly the modules  $\mathbf{K}_{w_0^\mathfrak{p}}T_x$,
where $x\in W$.
\end{proposition}

\begin{proof}
Since the derived functor of $\mathbf{K}_{w_0^\mathfrak{p}}$
is an equivalence, see \cite[Theorem~5.9]{MS3}, applying it
to various $T_x$ produces pairwise non-isomorphic 
indecomposable modules. So, we have the correct number $|W|$
of modules and we only need to check that they all are
tilting modules with respect to our structure. This means
that we need to verify that they all have both standard and
proper costandard filtrations.

To prove existence of a proper costandard filtration, due to
homological duality between standard and costandard modules,
we need to show that
\begin{equation}\label{eq-a1-1}
\mathrm{Ext}^i(\theta_u\Delta_w,\mathbf{K}_{w_0^\mathfrak{p}}T_x)=0, 
\end{equation}
for all $i>0$, all $u\in W(J)$ and all $w\in W_J$.
By adjunction, \eqref{eq-a1-1} is equivalent to 
\begin{equation}\label{eq-a1-2}
\mathrm{Ext}^i(\mathbf{C}_{w_0^\mathfrak{p}}\theta_u\Delta_w,T_x)=0. 
\end{equation}
Now we need the following auxiliary statement.

\begin{lemma}\label{lem-tilt2-1}
There is $\tilde{u}\in W(J)$ such that 
$\mathbf{C}_{w_0^\mathfrak{p}}\theta_u\cong
\theta_{\tilde{u}}\mathbf{C}_{w_0^\mathfrak{p}}.$
\end{lemma}

\begin{proof}
See Section~\ref{s5}.
\end{proof}

Using Lemma~\ref{lem-tilt2-1}, we can equivalently rewrite
\eqref{eq-a1-2} as 
\begin{equation}\label{eq-a1-3}
\mathrm{Ext}^i(\theta_{\tilde{u}}\mathbf{C}_{w_0^\mathfrak{p}}\Delta_w,T_x)=0. 
\end{equation}
Since $w\in W_J$, we have 
$\mathbf{C}_{w_0^\mathfrak{p}}\Delta_w=\Delta_{w_0^\mathfrak{p}w}$.
Consequently, $\theta_{\tilde{u}}\mathbf{C}_{w_0^\mathfrak{p}}\Delta_w
\cong \theta_{\tilde{u}}\Delta_{w_0^\mathfrak{p}w}$ has a Verma flag.
As Verma modules are homologically left orthogonal to dual Verma modules
and $T_x$, being tilting, has a dual Verma flag, we obtain
\eqref{eq-a1-3} and hence also \eqref{eq-a1-1}.
This establishes that $\mathbf{K}_{w_0^\mathfrak{p}}T_x$
has a proper costandard filtration.

Now we need to show that $\mathbf{K}_{w_0^\mathfrak{p}}T_x$
has a standard filtration, equivalently, that 
$\mathbf{K}_{w_0^\mathfrak{p}}T_x$ is left homologically
orthogonal to all proper costandard modules, that is, 
\begin{displaymath}
\mathrm{Ext}^i(\mathbf{K}_{w_0^\mathfrak{p}}T_x,
Q(w,u)^\star)=0, 
\end{displaymath}
for all $i>0$, $w\in W_J$ and $u\in W(J)$. We can, of course, move
$\mathbf{K}_{w_0^\mathfrak{p}}$, which we can change to 
$\mathcal{R}\mathbf{K}_{w_0^\mathfrak{p}}$
as $T_x$ has a dual Verma flag,
over to the other side to equivalently get 
\begin{equation}\label{eq-a1-4}
\mathrm{Ext}^i(T_x,
\mathcal{L}\mathbf{C}_{w_0^\mathfrak{p}}Q(w,u)^\star)=0.
\end{equation}
To prove the latter we will show that each complexes
$\mathcal{L}\mathbf{C}_{w_0^\mathfrak{p}}Q(w,u)^\star$
can be represented by a complex $\mathcal{T}^\bullet(w,u)$
of tilting modules such that $\mathcal{T}^i(w,u)=0$, for all
$i>0$. To do this, we start with the following ``local'' 
auxiliary result.

\begin{lemma}\label{lem-na6}
With respect to the usual highest weight structure of $\mathcal{O}_0$,
for $a,b\in W$ and $i>0$, we have
\begin{displaymath}
\mathrm{Ext}^i(\Delta_a,\mathcal{L}\mathbf{C}_{w_0}L_b)=0. 
\end{displaymath}
\end{lemma}

\begin{proof}
We have
\begin{displaymath}
\mathrm{Ext}^i(\Delta_a,\mathcal{L}\mathbf{C}_{w_0}L_b)=
\mathrm{Ext}^i(\mathcal{R}\mathbf{K}_{w_0}\Delta_a,L_b)
\end{displaymath}
as $\mathcal{L}\mathbf{C}_{w_0}$ and $\mathcal{R}\mathbf{K}_{w_0}$
are mutually inverse equivalences.

Being a Verma module, $\Delta_a$ has a tiling coresolution.
As $\mathbf{K}_{w_0}$ is acyclic on modules with dual Verma flag,
see \cite{MS2},
applying $\mathcal{R}\mathbf{K}_{w_0}$ to this coresolution is the 
same as applying $\mathbf{K}_{w_0}$. But the latter functor maps
tilting modules to projective modules. This means that 
$\mathcal{R}\mathbf{K}_{w_0}\Delta_a$ is represented by a complex
of projective modules and this complex is concentrated in non-negative
degrees.

At the same time, we can represent $L_b$ by its projective 
resolution, which is a complex of projective modules concentrated
in non-positive degrees. As we assume $i>0$, we get 
\begin{displaymath}
\mathrm{Ext}^i(\mathcal{R}\mathbf{K}_{w_0}\Delta_a,L_b)=0,
\end{displaymath}
implying the claim of the lemma.
\end{proof}

From \cite[Subsection~4.3]{MO2}, it follows that, in the setup 
of Lemma~\ref{lem-na6}, the module $\mathcal{L}\mathbf{C}_{w_0}L_b$
can be represented by a complex of tilting modules concentrated
in non-positive degrees. Now we want to export such local picture 
for $\mathfrak{a}$ to  the global picture for $\mathfrak{g}$.

Consider the Serre subquotient $\mathcal{B}_w$ of $\mathcal{O}_0$
as in \cite[Theorem~37]{CMZ}. Then $\mathcal{B}_w$ is equivalent to
the principal block of $\mathcal{O}_0(\mathfrak{a})$ and this
equivalence intertwines the action of $\theta_a$, for $a\in W(J)$.
Also, it maps Verma modules to Verma modules. In particular, 
the simple Verma module in  $\mathcal{B}_w$ is the image of
$\Delta_{ww_0^\mathfrak{p}}$. Thus the indecomposable 
tilting modules in  $\mathcal{B}_w$ are the images of 
$\theta_a \Delta_{ww_0^\mathfrak{p}}$, for $a\in W(J)$.
Moreover, the quotient map is full and faithful
(due to the equivalence in \cite[Theorem~37]{CMZ}).
So, we can lift any complex of tilting objects in 
$\mathcal{B}_w$ to a complex of modules over the form
$\theta_a \Delta_{ww_0^\mathfrak{p}}$. Using $\star$,
we can alternatively lift any complex of tilting objects in 
$\mathcal{B}_w$ to  a complex of modules
of the form $\theta_a \nabla_{ww_0^\mathfrak{p}}$.
Note that any module of the latter form has a
filtration by dual Verma modules.

Now consider the object 
$\mathcal{L}\mathbf{C}_{w_0^\mathfrak{p}} L_\mathfrak{a}(u\cdot \mu_w)$ in $\mathcal{B}_w$.
By Lemma~\ref{lem-na6}, it can be represented by a 
complex of tilting objects in $\mathcal{B}_w$ concentrated in
non-positive degrees. Lifting this as described in the previous paragraph,
we get a complex of modules in $\mathcal{O}_0$  in which each
component has a filtration by dual Verma
modules and this complex is concentrated in
non-positive degrees. By construction, this complex 
represents $\mathcal{L}\mathbf{C}_{w_0^\mathfrak{p}}Q(w,u)^\star$.
Since it lives in non-positive degrees, homomorphisms from tilting modules
concentrated in positive degrees to this complex vanish. 
This implies \eqref{eq-a1-4}
and completes the proof of our proposition.
\end{proof}

\section{Proof of Lemma~\ref{lem-tilt2-1}}\label{s5}

\subsection{Singular Artin monoid}\label{s5.1}

Consider the {\em singular Artin monoid} $SB$ associated to 
$W$, see \cite{J-Z} and references therein. If
$\mathtt{m}$ is the Coxeter matrix for $W$, the monoid
$SB$ has generators $s$ and $\underline{s}$, where $s$ runs over
all elements in $S$,
which are subject to the usual braid relations for all 
elements in $S$ combined with the following relations 
which involve $\underline{s}$: 
\begin{displaymath}
\underline{s}\,s=s\, \underline{s},\text{ for all }s\in S;
\end{displaymath}
\begin{displaymath}
\underline{s}\,\underline{t}=\underline{t}\,\underline{s},
\text{ for all }s,t\in S\text{ such that }\mathtt{m}(s,t)=2;
\end{displaymath}
\begin{displaymath}
\underbrace{\underline{s}\,ts\dots t}_{\mathtt{m}(s,t)\text{ factors}}=
\underbrace{ts\dots t\, \underline{s}}_{\mathtt{m}(s,t)\text{ factors}},
\text{ for all }s,t\in S\text{ such that }\mathtt{m}(s,t)\text{ is even};
\end{displaymath}
\begin{displaymath}
\underbrace{\underline{s}\,ts\dots s}_{\mathtt{m}(s,t)\text{ factors}}=
\underbrace{ts\dots s\, \underline{t}}_{\mathtt{m}(s,t)\text{ factors}},
\text{ for all }s,t\in S\text{ such that }\mathtt{m}(s,t)\text{ is odd}.
\end{displaymath}

\subsection{Commuting non-invertible generators with $w_0$}\label{s5.2}

Let $f$ be the permutation on $S$ defined via $w_0sw_0=f(s)$,
for $s\in S$. The aim of this subsection is to prove that 
$w_0\underline{s}w_0=\underline{f(s)}$, for all $s\in S$.
Unfortunately, we do not see any elegant argument for this,
so our proof will be a brute force case-by-case verification.
To prove the claim, it is enough to consider irreducible 
Weyl groups.
In the remainder of this subsection, we go over all irreducible 
Weyl groups, type by type. We use \cite[Table~1]{BKOP} for
explicit reduced expressions for $w_0$ in different types.
We can also assume that the rank of $W$ is at least $3$,
for, if the rank is $2$, the claim follows directly from the
defining relations for $SB$.

\subsubsection{\bf Type $A$}\label{s5.2.1}

In type $A$, we assume that our Dynkin diagram is as follows:
\begin{displaymath}
\xymatrix{s_1\ar@{-}[rr]&&s_2\ar@{-}[rr]&&\dots\ar@{-}[rr]&&s_k} 
\end{displaymath}
Then $w_0=w_0^{(k)}:=s_1(s_2s_1)(s_3s_2s_1)\dots (s_ks_{k-1}\dots s_2s_1)$
and $f(s_i)=s_{k+1-i}$, for $i=1,2,\dots,k$. We proceed by induction 
on $k$, with the basis of the induction, given by the cases $k=1,2$, 
being trivial.

To prove the induction step, we start with $i\leq k$. Then, using the
inductive assumption and the observation that 
$w_0^{(k+1)}=w_0^{(k)}(s_{k+1}s_{k}\dots s_2s_1)$, we have
\begin{displaymath}
\underline{s_i}w_0^{(k+1)}=
w_0^{(k)}\underline{s_{k+1-i}}(s_{k+1}s_{k}\dots s_2s_1).
\end{displaymath}
If we denote $k+1-i$ by $j$, we can use the defining relation to move
$\underline{s_{j}}$ past a prefix of $s_{k+1}s_{k-1}\dots s_2s_1$
all the way to $s_{j+1}s_j$, then we use
$\underline{s_{j}}s_{j+1}s_j=s_{j+1}s_j\underline{s_{j+1}}$
and then we can move $\underline{s_{j+1}}$ past the remaining suffix.
As $j+1=k+2-i$, this case is done.

It remains to consider the case $i=k+1$. In this case, we use
commutation to obtain 
\begin{displaymath}
\underline{s_{k+1}}w_0^{(k+1)}=
w_0^{(k-1)}\underline{s_{k+1}}(s_{k}s_{k-1}\dots s_2s_1)(s_{k+1}s_{k}\dots s_2s_1).
\end{displaymath}
Rewrite the suffix starting with 
$\underline{s_{k+1}}$ on the right hand side of this equality as follows:
\begin{equation}\label{eq-nm3}
\underline{s_{k+1}}(s_{k}s_{k+1})(s_{k-1}s_k)\dots (s_1s_2)s_1.
\end{equation}
Now we use the relations, first we have:
\begin{displaymath}
(s_{k}s_{k+1})\underline{s_{k}}(s_{k-1}s_k)\dots (s_1s_2)s_1,
\end{displaymath}
then
\begin{displaymath}
(s_{k}s_{k+1})(s_{k-1}s_k)\underline{s_{k-1}}\dots (s_1s_2)s_1,
\end{displaymath}
and, proceeding inductively, we obtain
\begin{displaymath}
(s_{k}s_{k+1})(s_{k-1}s_k)\dots (s_1s_2)\underline{s_{1}}s_1.
\end{displaymath}
Here, using $\underline{s_{1}}s_1=s_1\underline{s_{1}}$,
we, finally,  obtain $\underline{s_{1}}$ at the rightmost place.
This completes the proof in type $A$.

\subsubsection{\bf Types $B$  and $C$}\label{s5.2.2}

In types $B$ and $C$, we assume that our Dynkin diagram is
obtained by orienting the double edge in the following:
\begin{displaymath}
\xymatrix{s_1\ar@{-}[rr]&&s_2\ar@{-}[rr]&&\dots\ar@{-}[rr]&&
s_{k-1}\ar@{=}[rr]&&s_k} 
\end{displaymath}
Note that the Weyl group and its presentation
do not depend on this orientation.
Then $w_0=w_0^{(k)}$ is given by 
\begin{displaymath}
s_k(s_{k-1}s_ks_{k-1})
(s_{k-2}s_{k-1}s_ks_{k-1}s_{k-2})\dots (s_1s_2\dots 
s_{k-1}s_ks_{k-1}\dots s_2s_1), 
\end{displaymath}
and $f(s_i)=s_{i}$, for $i=1,2,\dots,k$. We proceed by induction 
on $k$, with the basis $k=2$ of the induction being trivial.

To prove the induction step, we start with $1<i<k+1$. Then, using the
inductive assumption and the observation that 
$w_0^{(k+1)}=w_0^{(k)}(s_1\dots s_{k+1}\dots s_1)$, we have
\begin{displaymath}
\underline{s_i}w_0^{(k+1)}=
w_0^{(k)}\underline{s_i}(s_1\dots s_{k+1}\dots s_1).
\end{displaymath}
We now move $\underline{s_i}$ through the commuting part of the prefix
to $\underline{s_i}s_{i-1}s_i$, then use
$\underline{s_i}s_{i-1}s_i=s_{i-1}s_i\underline{s_{i-1}}$, then move
$\underline{s_{i-1}}$ all the way to 
$\underline{s_{i-1}}s_is_{i-1}$, then use the relation
$\underline{s_{i-1}}s_is_{i-1}=s_is_{i-1}\underline{s_i}$ and now can
move $\underline{s_i}$ out to the right via the commuting part of the suffix.

Next consider the case $i=1$. In this case, using commutativity, we have
\begin{displaymath}
\underline{s_1}w_0^{(k+1)}=
w_0^{(k-1)}\underline{s_1}(s_2\dots s_{k+1}\dots s_2)(s_1\dots s_{k+1}\dots s_1).
\end{displaymath}
Let us now focus on the part to the right of $\underline{s_1}$, 
which we call $w$. There is
a segment $s_2s_1s_2$ in the middle, which we rewrite as $s_1s_2s_1$ and move
the left $s_1$ from here all the way to the left in our expression
for $w$. This gives us $\underline{s_1}s_2s_1$ at the very left.
Using $\underline{s_1}s_2s_1=s_2s_1\underline{s_2}$, we now get the
subword $\underline{s_2}s_3$. Proceeding inductively, we eventually obtain
\begin{displaymath}
(s_2s_1)(s_3s_2)\dots (s_ks_{k-1})
(\underline{s_k}s_{k+1}s_ks_{k+1})(s_{k-1}\dots s_1s_k\dots s_1).
\end{displaymath}
Here we use $\underline{s_k}s_{k+1}s_ks_{k+1}=s_{k+1}s_ks_{k+1}\underline{s_k}$
to get $\underline{s_k}(s_{k-1}\dots s_1s_k\dots s_1)$. We write
$(s_{k-1}\dots s_1s_k\dots s_1)$ as 
$(s_{k-1}s_k)(s_{k-2}s_{k-1})\dots (s_1s_2)s_1$.
Now, we can use the relation
$\underline{s_k}(s_{k-1}s_k)=(s_{k-1}s_k)\underline{s_{k-1}}$,
then $\underline{s_{k-1}}(s_{k-2}s_{k-1})=(s_{k-2}s_{k-1})\underline{s_{k-2}}$
and so on all the way to $\underline{s_{1}}s_1$. Here, we finally use
$\underline{s_{1}}s_1=s_1\underline{s_{1}}$ and we are done.

Finally, consider the case $i=k+1$. In this case, using the inductive
assumption, we have
\begin{displaymath}
\underline{s_{k+1}}w_0^{(k+1)}=
w_0^{(k)}\underline{s_{k+1}}(s_1\dots s_{k+1}\dots s_1).
\end{displaymath}
Using commutativity, we can move $\underline{s_{k+1}}$ all the way to
$\underline{s_{k+1}}s_{k}s_{k+1}s_{k}$, then we commute it past
$s_{k}s_{k+1}s_{k}$ and then move out all the way to the extreme right.
This completes the proof in types $B$ and $C$.

\subsubsection{\bf Type $D$}\label{s5.2.3}

In type $D$, we assume that our Dynkin diagram is:
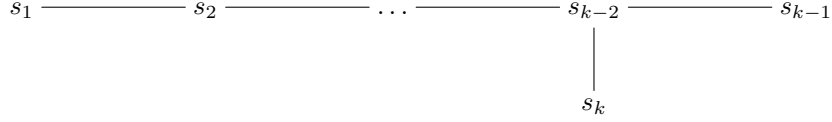
\begin{displaymath}
\xymatrix{s_1\ar@{-}[rr]&&s_2\ar@{-}[rr]&&\dots\ar@{-}[rr]&&
s_{k-2}\ar@{-}[rr]\ar@{-}[d]&&s_{k-1}\\&&&&&&s_k&&} 
\end{displaymath}
Then $w_0=w_0^{(k)}$ is given, for $k\geq 2$, by 
\begin{multline*}
s_ks_{k-1}(s_{k-2}s_ks_{k-1}s_{k-2})
(s_{k-3}s_{k-2}s_ks_{k-1}s_{k-2}s_{k-3})\dots\\ \dots  (s_1s_2\dots 
s_{k-2}s_ks_{k-1}s_{k-2}\dots s_2s_1),
\end{multline*}
and 
\begin{displaymath}
f(s_i)=
\begin{cases}
s_i,& k\text{ is even};\\
s_i,& i\leq k-2\text{ and }k \text{ is odd};\\
s_k, & i=k-1\text{ and }k \text{ is odd};\\
s_{k-1}, & i=k\text{ and }k \text{ is odd}.
\end{cases}
\end{displaymath} 
As usual, we proceed by induction on $k$, with the basis given by
the cases $k=1,2,3$ being already dealt with as type $A$.

So, let $k>3$ and take $2\leq i \leq k-2$. Then, using the inductive
assumption and the observation that 
$w_0^{(k+1)}=w_0^{(k)}(s_1s_2\dots 
s_{k-2}s_ks_{k-1}s_{k-2}\dots s_2s_1)$, we can write 
\begin{displaymath}
\underline{s_i}w_0^{(k+1)}=
w_0^{(k)}\underline{s_i}(s_1s_2\dots 
s_{k-2}s_ks_{k-1}s_{k-2}\dots s_2s_1).
\end{displaymath}
We use commutativity to move $\underline{s_i}$ all the way until it hits
$s_{}i-1$ giving $\underline{s_i}s_{i-1}s_i$, then we use 
$\underline{s_i}s_{i-1}s_i=s_{i-1}s_i\underline{s_{i-1}}$
and now we can move $\underline{s_{i-1}}$ all the way to
$\underline{s_{i-1}}s_is_{i-1}$. Using 
$\underline{s_{i-1}}s_is_{i-1}=s_is_{i-1}\underline{s_i}$,
we can now move $\underline{s_i}$ to the extreme right and we are done.

Next consider the cases $i=k$ or  $i=k-1$, which are equivalent, by symmetry.
From the inductive assumption and using the parity considerations,
we just need to check that, say, commuting $\underline{s_{k-1}}$
past $s_1s_2\dots 
s_{k-2}s_{k-1}s_ks_{k-2}\dots s_2s_1$ outputs 
$\underline{s_{k}}$ on the other side.
We use commutativity and move $\underline{s_{k-1}}$ until we get
$\underline{s_{k-1}}s_{k-2}s_{k-1}$. Now we use
$\underline{s_{k-1}}s_{k-2}s_{k-1}=
s_{k-2}s_{k-1}\underline{s_{k-2}}$ and get
$\underline{s_{k-2}}s_ks_{k-2}$. We replace this with
$s_ks_{k-2}\underline{s_{k}}$ and now we can move out
$\underline{s_{k}}$ all the way to the extreme right.

It remains to consider the case $i=1$. Here we can use commutativity
to obtain 
\begin{displaymath}
\underline{s_1}w_0^{(k+1)}=
w_0^{(k-2)}\underline{s_{1}}
(s_2\dots 
s_{k-2}s_{k-1}s_ks_{k-2}\dots s_2)
(s_1\dots 
s_{k-2}s_{k-1}s_ks_{k-2}\dots s_1).
\end{displaymath}
As above, look at the part to the right of $\underline{s_{1}}$.
We observe the $s_2s_1s_2$ subword in the middle.
We rewrite it as $s_1s_2s_1$ and move the right $s_1$ all the way to the left
to get $\underline{s_{1}}s_2s_1$. This we can rewrite as
$s_2s_1\underline{s_{2}}$. Playing a similar game over and over, we
eventually get 
\begin{displaymath}
\underline{s_{k-1}}(s_ks_{k+1})(s_{k-1}s_{k-2}\dots s_1)
(s_ks_{k+1})(s_{k-1}s_{k-2}\dots s_1).
\end{displaymath}
Now we move the middle $s_k$ and $s_{k+1}$ (which, by the way, commute)
to the left all the way to get 
$(s_ks_{k+1})s_{k-1}(s_ks_{k+1})=(s_{k+1}s_k)s_{k-1}(s_ks_{k+1})$.
We rewrite this as $s_{k+1}s_{k-1}s_ks_{k-1}s_{k+1}$ and now we can use
$\underline{s_{k-1}}s_{k+1}s_{k-1}=s_{k+1}s_{k-1}\underline{s_{k+1}}$
and further move $\underline{s_{k+1}}$ past $s_k$ as they commute.
This gives $\underline{s_{k+1}}s_{k-1}s_{k+1}$ which we rewrite as
$s_{k-1}s_{k+1}\underline{s_{k-1}}$ giving
\begin{displaymath}
\underline{s_{k-1}}(s_{k-2}\dots s_1) (s_{k-1}s_{k-2}\dots s_1).
\end{displaymath}
This is now a pure type $A$ picture. We rewrite this as
\begin{displaymath}
\underline{s_{k-1}}(s_{k-2}s_{k-1})
(s_{k-3}s_{k-2}) \dots (s_1s_2)s_1
\end{displaymath}
and get a situation similar to \eqref{eq-nm3}. Using the same 
procedure as in that case, we eventually get $\underline{s_{1}}$
on the extreme right. This completes the proof in type $D$.

\subsubsection{\bf Type $F_4$}\label{s5.2.4}

In type $F_4$, we assume that our Dynkin diagram is:
\begin{displaymath}
\xymatrix{s_1\ar@{-}[rr]&&s_2\ar@{=>}[rr]&&
s_{3}\ar@{-}[rr]&&s_{4}} 
\end{displaymath}
Then $w_0$ is given by 
\begin{displaymath}
w_0=w'_0(s_4s_3s_2s_1s_3s_2s_3s_4s_3s_2s_1s_3s_2s_3s_4),
\end{displaymath}
where $w'_0$ is the longest element in type $B_3$.
Also $f$ is the identity. From type $B$, we also know that 
$\underline{s_{i}}w'_0=w'_0\underline{s_{i}}$, for $i=1,2,3$.
So, for these three elements we just need to show that they
commute with the additional factor 
$s_4s_3s_2s_1s_3s_2s_3s_4s_3s_2s_1s_3s_2s_3s_4$.

Let us start with $\underline{s_{1}}$. We use commutativity to move it
past $s_4s_3$, then we use that $\underline{s_{1}}s_2s_1=s_2s_1\underline{s_{2}}$.
Then $\underline{s_{2}}$ commutes with $s_3s_2s_3$ and then with
$s_4$, so we get 
\begin{displaymath}
\underline{s_{2}}s_3s_2s_1s_3s_2s_3s_4.
\end{displaymath}
We swap the commuting $s_1s_3$ in the middle and move
$\underline{s_{2}}$ past $s_3s_2s_3$ to get
$\underline{s_{2}}s_1s_2s_3s_4$. Here we use
$\underline{s_{2}}s_1s_2=s_1s_2\underline{s_{1}}$ and now we can move
$\underline{s_{1}}$ out to the extreme right.

Next we move $\underline{s_{2}}$ through
$s_4s_3s_2s_1s_3s_2s_3s_4s_3s_2s_1s_3s_2s_3s_4$.
It commutes with $s_4$ and then we can swap $s_1s_3$
to $s_3s_1$ in the leftmost occurrence to get
$\underline{s_{2}}s_3s_2s_3$, which we rewrite as
$s_3s_2s_3\underline{s_{2}}$ to get 
\begin{displaymath}
\underline{s_{2}}s_1s_2s_3s_4s_3s_2s_1s_3s_2s_3s_4.
\end{displaymath}
We replace $\underline{s_{2}}s_1s_2$ by 
$s_1s_2\underline{s_{1}}$ and can commute $\underline{s_{1}}$
with $s_3s_4s_3$ to get $\underline{s_{1}}s_2s_1s_3s_2s_3s_4$.
Now, replace $\underline{s_{1}}s_2s_1$ with
$s_2s_1\underline{s_{2}}$ and move $\underline{s_{2}}$
first through $s_3s_2s_3$ and the through $s_4$, and we are done.

The above prove the claim for $s_1$ and $s_2$. For
$s_3$ and $s_4$, the claim follows by the symmetry of the 
Weyl group. This completes the proof in type $F_4$.

\subsubsection{\bf Type $E_6$}\label{s5.2.5}

In type $E_6$, we assume that our Dynkin diagram is:
\begin{displaymath}
\xymatrix{s_1\ar@{-}[rr]&&s_2\ar@{-}[rr]&&
s_{3}\ar@{-}[rr]\ar@{-}[d]&&s_{5}\ar@{-}[rr]&&s_{6}\\
&&&&s_{4}&&&&} 
\end{displaymath}
Then $w_0$ is given by 
\begin{displaymath}
w_0=w'_0(s_6 s_5 s_3 s_4 s_2 s_1 s_3 s_2 s_5 s_3 s_4 s_6 s_5 s_3 s_2 s_1),
\end{displaymath}
where $w'_0$ is the longest element in type $D_5$.
Also $f$ is given by:
\begin{displaymath}
\begin{array}{c||c|c|c|c|c|c}
i&1&2&3&4&5&6\\
\hline
f(i)&6&5&3&4&2&1
\end{array}
\end{displaymath}
From type $D_5$, we also know that 
$\underline{s_{i}}w'_0=w'_0\underline{s_{i}}$, for $i=1,2,3$,
and, additionally, we have  
$\underline{s_{4}}w'_0=w'_0\underline{s_{5}}$
as well as $\underline{s_{5}}w'_0=w'_0\underline{s_{4}}$.
Due to symmetry, we only need to check the cases $i=1,2,3,4$
and in these cases try to commute the outcome of the 
commutation of  $\underline{s_{i}}$ with $w'_0$
with the additional factor 
$s_6 s_5 s_3 s_4 s_2 s_1 s_3 s_2 s_5 s_3 s_4 s_6 s_5 s_3 s_2 s_1$.
Let us do that.

We start with $\underline{s_{1}}$, which we can immediately pass
through $s_6 s_5 s_3 s_4$ and then change 
$\underline{s_{1}}s_2 s_1$  to $s_2 s_1 \underline{s_{2}}$.
Next we change $\underline{s_{2}} s_3 s_2$ to
$s_3 s_2 \underline{s_{3}}$, next 
$\underline{s_{3}} s_5 s_3$ to $ s_5 s_3\underline{s_{5}}$.
Now we commute $\underline{s_{5}}$ past $s_4$ and change
$\underline{s_{5}}s_6 s_5$ to $s_6 s_5\underline{s_{6}}$.
Finally, we move $\underline{s_{6}}$ out on the right. 

Next we do $\underline{s_{2}}$, which we pass
through $s_6 s_5$. Note that $s_4$ and $s_2$ commute, so
we swap them and replace 
$\underline{s_{2}}s_3 s_2$  by $s_3 s_2 \underline{s_{3}}$.
We swap $s_4$ and $s_1$ and move  $\underline{s_{3}}$
past $s_1$. Now we replace 
$\underline{s_{3}}s_4 s_3$  by $s_4 s_3 \underline{s_{4}}$
and commute $\underline{s_{4}}$ with $s_2 s_5$.
Now $\underline{s_{4}}s_3 s_4$ is replaced by $s_3 s_4 \underline{s_{3}}$,
and then $\underline{s_{3}}$ moves through $s_6$.
Finally, $\underline{s_{3}} s_5 s_3$ equals
$s_5 s_3\underline{s_{5}}$ and we can move $\underline{s_{5}}$
out past $s_2 s_1$.

Next we do $\underline{s_{3}}$, which we commute with $s_6$, to start with.
Then $\underline{s_{3}} s_5 s_3$ is changed to $s_5 s_3\underline{s_{5}}$
and $\underline{s_{5}}$ commutes with $s_4 s_2 s_1$. Now we swap 
$s_2$ and $s_5$ and replace $\underline{s_{5}} s_3 s_5$ with
$s_3 s_5 \underline{s_{3}}$. After that, we change 
$\underline{s_{3}} s_2 s_3 $ to $s_2 s_3 \underline{s_{2}}$ and, in turn,
$\underline{s_{2}}$ commutes with $s_4 s_6 s_5$. Finally,
$\underline{s_{2}}s_3 s_2$ changes to $s_3 s_2\underline{s_{3}}$
and we move $\underline{s_{3}}$ out past $s_1$.

Our last element is $\underline{s_{4}}$ and we recall that
$\underline{s_{4}}w'_0=w'_0\underline{s_{5}}$. So, we try to 
move $\underline{s_{5}}$ through our additional element
$s_6 s_5 s_3 s_4 s_2 s_1 s_3 s_2 s_5 s_3 s_4 s_6 s_5 s_3 s_2 s_1$.
To start with, $\underline{s_{5}}s_6 s_5$ changes to
$s_6 s_5\underline{s_{6}}$. Now $\underline{s_{6}}$ commutes
with $s_3 s_4 s_2 s_1 s_3 s_2$. Now we note that $s_3s_4$
commutes with $s_6$, so we swap them and replace
$\underline{s_{6}}s_5 s_6$ with $s_5 s_6\underline{s_{5}}$.
Next, we swap $s_4$ and $s_5$ and replace 
$\underline{s_{5}}s_3 s_5$ with $s_3 s_5\underline{s_{3}}$.
Finally, $\underline{s_{3}} s_4  s_3$ changes to
$ s_4  s_3\underline{s_{4}}$ and we can move
$\underline{s_{4}}$ to the extreme right past $s_2 s_1$.
This completes the proof in type $E_6$.

\subsubsection{\bf Type $E_7$}\label{s5.2.6}

In type $E_7$, we assume that our Dynkin diagram is:
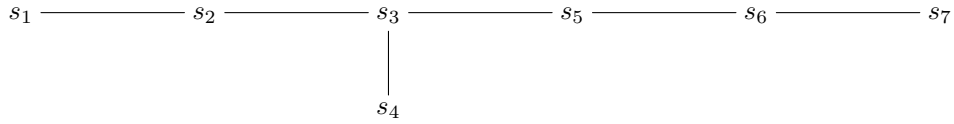
\begin{displaymath}
\xymatrix{s_1\ar@{-}[rr]&&s_2\ar@{-}[rr]&&
s_{3}\ar@{-}[rr]\ar@{-}[d]&&s_{5}\ar@{-}[rr]&&s_{6}\ar@{-}[rr]&&s_{7}\\
&&&&s_{4}&&&&} 
\end{displaymath}
Then $w_0$ is given by 
\begin{displaymath}
w_0=w'_0(s_7 s_6 s_5 s_3 s_4 s_2 s_1 s_3 s_2 
s_5 s_3 s_4 s_6 s_5 s_3 s_2 s_1 s_7 s_6 s_5 s_3 s_4 s_2
s_3 s_5 s_6 s_7),
\end{displaymath}
where $w'_0$ is the longest element in type $E_6$.
Also $f$ is the identity. From type $E_6$, we also know that 
$\underline{s_{i}}w'_0=w'_0\underline{s_{i}}$, for $i=3,4$,
and, additionally, we have  
$\underline{s_{1}}w'_0=w'_0\underline{s_{6}}$,
$\underline{s_{6}}w'_0=w'_0\underline{s_{1}}$,
$\underline{s_{2}}w'_0=w'_0\underline{s_{5}}$
as well as $\underline{s_{5}}w'_0=w'_0\underline{s_{2}}$.
Now we need to commute all $\underline{s_{i}}$ past the
additional factor 
$s_7 s_6 s_5 s_3 s_4 s_2 s_1 s_3 s_2 
s_5 s_3 s_4 s_6 s_5 s_3 s_2 s_1 s_7 s_6 s_5 s_3 s_4 s_2
s_3 s_5 s_6 s_7$.
After that, the element $\underline{s_{7}}$ has to be 
considered separately.

Let us start with $\underline{s_{6}}$. Commuting with 
$w'_0$ gives $\underline{s_{1}}$. Now, $\underline{s_{1}}$
commutes with $s_7 s_6 s_5 s_3 s_4$ and then changes to
$\underline{s_{2}}$ passing through $s_2 s_1$. This
changes to $\underline{s_{3}}$ passing through $s_3 s_2$
and then to $\underline{s_{5}}$ passing through $s_5 s_3$.
Then $\underline{s_{5}}$ commutes with $s_4$ and 
changes to $\underline{s_{6}}$ passing through $s_3 s_2 s_1$.
Then $\underline{s_{6}}$ commutes with $s_4$ and 
changes to $\underline{s_{7}}$ passing through $ss_7 s_6$.
Finally, $\underline{s_{7}}$ commutes with 
$s_5 s_3 s_4 s_2s_3 s_5$ and 
changes to $\underline{s_{6}}$ passing through $s_6 s_7$.
So, we have checked that $\underline{s_{6}}$ commutes with $w_0$.

Next we do $\underline{s_{5}}$. Commuting with 
$w'_0$ gives $\underline{s_{2}}$. Now, $\underline{s_{2}}$
commutes with $s_7 s_6 s_5$. We swap $s_2$ and $s_4$
after $s_3$, since they commute. And then we can change 
$\underline{s_{2}}$ to $\underline{s_{3}}$ passing through $s_3 s_2$.
We swap $s_1$ and $s_3$
after $s_4$, since they commute. And then we can change 
$\underline{s_{3}}$ to $\underline{s_{4}}$ passing through $s_4 s_3$.
Now, $\underline{s_{4}}$ commutes with $s_1 s_2  s_5$ and then 
changes to $\underline{s_{3}}$ passing through $s_3 s_4$.
Then, $\underline{s_{3}}$ commutes with $s_6$ and then 
changes to $\underline{s_{5}}$ passing through $s_5 s_3$.
In turn, $\underline{s_{5}}$ commutes with $s_2 s_1 s_7 $ and then 
changes to $\underline{s_{6}}$ passing through $s_6 s_5$.
Next, $\underline{s_{6}}$ commutes with $s_3 s_4 s_2 s_3$ and then 
changes to $\underline{s_{5}}$ passing through $s_5 s_6$.
Finally, we can pull this $\underline{s_{5}}$ out through $s_7$.
So, we have checked that $\underline{s_{5}}$ commutes with $w_0$.

Next we do $\underline{s_{4}}$. Commuting with 
$w'_0$ gives $\underline{s_{4}}$. Now, $\underline{s_{4}}$
commutes with $s_7 s_6 s_5$ and changes to  $\underline{s_{3}}$
passing through $s_3 s_4$. We swap $s_1$ and $s_3$
after $S_2$ since they commute and now we can move
$\underline{s_{3}}$ through $s_2 s_3$ changing it to $\underline{s_{2}}$.
Further, $\underline{s_{2}}$ changes to $\underline{s_{1}}$
passing through $s_1s_2$. This $\underline{s_{1}}$
commutes with $s_5 s_3 s_4 s_6 s_5 s_3$ and changes to $\underline{s_{2}}$
passing through $s_2s_1$. The $\underline{s_{2}}$
commutes with $s_7 s_6 s_5$ and changes to $\underline{s_{3}}$
passing through $s_3s_2$ (here we swapped $s_2$ and $s_4$
after $s_3$ as they commute). Finally, we change 
$\underline{s_{3}}$ to $\underline{s_{4}}$ when it passes
through $s_4 s_3$ and pull it out on the right through
$s_5 s_6 s_7$. We have checked that 
$\underline{s_{4}}$ commutes with $w_0$.

Next we do $\underline{s_{3}}$. Commuting with 
$w'_0$ gives $\underline{s_{3}}$. Now, $\underline{s_{3}}$
commutes with $s_7 s_6$ and changes to  $\underline{s_{5}}$
passing through $s_5 s_3$.  This $\underline{s_{5}}$ commues
with $s_4 s_2 s_1$. Now we swap $s_2$ and $s_5$ after $s_3$
as they commute and move $\underline{s_{5}}$ through 
$s_3s_5$ changing it to $\underline{s_{3}}$. Next we move
$\underline{s_{3}}$ through $s_2 s_3$ changing it to 
$\underline{s_{2}}$. Now $\underline{s_{2}}$ commutes with
$s_4 s_6 s_5$ and changes to $\underline{s_{3}}$ after
passing through $s_3 s_2$. This $\underline{s_{3}}$ commutes
with $s_1 s_7 s_6$ and changes to $\underline{s_{5}}$
after passing through $s_5 s_3$. Finally, this 
$\underline{s_{5}}$ commutes
with $s_4 s_2$ and changes to $\underline{s_{3}}$
after passing through $s_3 s_5$ and then we can pull
it out on the right via $s_6 s_7$. We have checked that 
$\underline{s_{3}}$ commutes with $w_0$.

Out next step is $\underline{s_{2}}$. Commuting with 
$w'_0$ gives $\underline{s_{5}}$. Now, $\underline{s_{5}}$
commutes with $s_7 $ and changes to  $\underline{s_{6}}$
passing through $s_6 s_5$.  This $\underline{s_{6}}$ commutes
with $s_3 s_4 s_2 s_1 s_3 s_2$. We swap $s_3 s_4$ with 
$s_6$ as they commute and now we can move $\underline{s_{6}}$
through $s_5 s_6$ changing it to $\underline{s_{5}}$.
after passing through $s_5 s_3$. We swap $s_4$ and $s_5$
after $s_3$ as they commute and now we can move 
$\underline{s_{5}}$ through $s_3 s_5$ changing it to
$\underline{s_{3}}$. Next we move $\underline{s_{3}}$
through $s_4 s_3$ changing it to $\underline{s_{3}}$.
This $\underline{s_{4}}$ commutes
with $s_2 s_1 s_7 s_6 s_5$ and changes to $\underline{s_{3}}$
when passing through $s_3 s_4$. Finally, this 
$\underline{s_{3}}$ changes to $\underline{s_{2}}$
when passing through $s_2 s_3$ and we can pull it out on
the right through $s_3 s_5 s_6 s_7$. We have checked that 
$\underline{s_{2}}$ commutes with $w_0$.

Out next step is $\underline{s_{1}}$.  Commuting with 
$w'_0$ gives $\underline{s_{6}}$ and it changes to 
$\underline{s_{7}}$ passing through $s_7 s_6$.
This $\underline{s_{7}}$ commutes with
$s_5 s_3 s_4 s_2 s_1 s_3 s_2 
s_5 s_3 s_4$. Now we swap $s_5 s_3 s_2 s_1$
and $s_7$ after $s_6$ and change
$\underline{s_{7}}$ to $\underline{s_{6}}$
after passing through $s_6 s_7$.
Next we swap $s_3 s_2 s_1$
and $s_6$ after $s_5$ and change
$\underline{s_{6}}$ to $\underline{s_{5}}$
after passing through $s_5 s_6$.
Next we swap $s_2 s_1$
and $s_5$ after $s_3$ and change
$\underline{s_{5}}$ to $\underline{s_{3}}$
after passing through $s_3 s_5$.
Next we swap $s_1$
and $s_3$ after $s_2$ and change
$\underline{s_{3}}$ to $\underline{s_{2}}$
after passing through $s_2 s_3$.
Finally,  we swap $s_4$
and $s_2$ after $s_1$ and change
$\underline{s_{2}}$ to $\underline{s_{1}}$
after passing through $s_1 s_2$ and then
pull it out on the right through $s_4 s_3 s_5 s_6 s_7$.
We have checked that 
$\underline{s_{1}}$ commutes with $w_0$.

It remains to check $\underline{s_{7}}$. Let
$w''_0$ be the longest element for type $D_5$.
Then we have $\underline{s_{7}}w''_0=w''_0\underline{s_{7}}$
due to commutativity. Now we have to move 
$\underline{s_{7}}$ through the product of the additional
factor in type $E_6$ followed by the additional factor in
type $E_7$. To save space, we write this product as a sequence
of indices:
\begin{displaymath}
6534213253465321765342132534653217653423567
\end{displaymath}
We need to  move $\underline{{7}}$ through it from left to right.
For this, we will need some help from SageMath. Below we use SageMath
to compute reduced expressions and descent sets.

By construction, our element is $w''_0w_0$ and hence its left descent set
is $\{6,7\}$. According to, SageMath our element equals
\begin{displaymath}
767 * 5342356712356435234123567235643523412356.
\end{displaymath}
Note that $\underline{{7}}$ commutes with $7$ and then changes to
$\underline{{6}}$ passing through $67$. The left descent set of
the remaining element is just $\{5\}$ and we can write it as 
\begin{displaymath}
56 *  34235671235643523412356735643523412356
\end{displaymath}
Now $\underline{{6}}$ changes to
$\underline{5}$ passing through $56$. The left descent set of
the remaining element is just $\{3\}$ and we can write it as 
\begin{displaymath}
35 *  423567123564352341235675643523412356
\end{displaymath}
Now $\underline{{5}}$ changes to
$\underline{3}$ passing through $35$. The left descent set of
the remaining element is just $\{2,4\}$ and we can write it as 
\begin{displaymath}
43 *  2356712356435234123567563523412356
\end{displaymath}
Now $\underline{{3}}$ changes to
$\underline{4}$ passing through $43$. The left descent set of
the remaining element is $\{2\}$ and it commutes with 
$\underline{4}$. Removing the leftmost $2$, we get an element 
with the left descent set $\{1,3\}$ and $\underline{4}$ commutes
with $1$. Removing this $1$, we get an element 
with the left descent set $\{3\}$ which we can write
\begin{displaymath}
34 *  567235643523412356763523412356
\end{displaymath}
and $\underline{4}$ changes to $\underline{3}$ moving through
$34$. The left descent set of
the remaining element is just $\{2,5\}$ and we can write it as 
\begin{displaymath}
23 *  5673564352341235673523412356
\end{displaymath}
Now $\underline{{3}}$ changes to
$\underline{2}$ passing through $23$.  The left descent set of
the remaining element is $\{5\}$ and it commutes with 
$\underline{2}$. Removing the leftmost $5$, we get an element 
with the left descent set $\{3,6\}$ and $\underline{2}$ commutes
with $6$. Removing the leftmost $6$, we get an element 
with the left descent set $\{3,7\}$ and $\underline{2}$ commutes
with $7$. Removing the leftmost $6$, we get an element 
with the left descent set $\{3\}$ which we write
\begin{displaymath}
32 * 56435234123567523412356
\end{displaymath}
Now $\underline{{2}}$ changes to
$\underline{3}$ passing through $32$. 
The left descent set of
the remaining element is just $\{4,5\}$ and we can write it as 
\begin{displaymath}
43 *  563523412356723412356
\end{displaymath}
Now $\underline{{3}}$ changes to
$\underline{4}$ passing through $43$.
Then, $\underline{4}$ commutes with $56$.
Removing $56$, we get an element with 
left descent set $\{3\}$ and we write it
\begin{displaymath}
34 *  52341235673412356
\end{displaymath}
Now $\underline{{4}}$ changes to
$\underline{3}$ passing through $34$.
The left descent set of
the remaining element is just $\{2,5\}$ and we can write it as 
\begin{displaymath}
53 *  234123567342356
\end{displaymath}
Now $\underline{{3}}$ changes to
$\underline{5}$ passing through $53$.
Then $\underline{5}$ commutes with $2$. 
The left descent set of
the remaining element is just $\{1,3\}$
so we can commute $\underline{5}$ with $1$.
The left descent set of
the remaining element is just $\{3\}$ and we can write it as 
\begin{displaymath}
35 *  42356742356
\end{displaymath}
Now $\underline{{5}}$ changes to
$\underline{3}$ passing through $35$.
The left descent set of
the remaining element is just $\{2,4\}$ and we can write it as 
\begin{displaymath}
43 *  235674356
\end{displaymath}
Now $\underline{{3}}$ changes to
$\underline{4}$ passing through $43$
and this $\underline{4}$ commutes with $2$.
We write  $35674356$ as $34567356$ and 
$\underline{{4}}$ changes to
$\underline{3}$ passing through $34$.
We write $567356$  as $536756$ and 
$\underline{{3}}$ changes to
$\underline{5}$ passing through $53$.
We write $6756$  as $6576$ and 
$\underline{{5}}$ changes to
$\underline{6}$ passing through $56$.
Finally, $\underline{{6}}$ changes to
$\underline{7}$ passing through $76$.
We are done!

\subsubsection{\bf Type $E_8$}\label{s5.2.7}

In type $E_8$, we assume that our Dynkin diagram is:

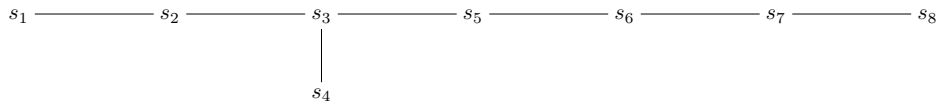
\resizebox{\textwidth}{!}{
$
\xymatrix{s_1\ar@{-}[rr]&&s_2\ar@{-}[rr]&&
s_{3}\ar@{-}[rr]\ar@{-}[d]&&s_{5}\ar@{-}[rr]
&&s_{6}\ar@{-}[rr]&&s_{7}\ar@{-}[rr]&&s_{8}\\
&&&&s_{4}&&&&} 
$
}

The element $w_0$ is given by $w_0=w'_0u$, 
where $w'_0$ is the longest element in type $E_7$
and $u$ is the product of the $s_i$'s in which the indices
go along the following word:
\begin{displaymath}
876534213253465321765342356787653421325346532176534235678
\end{displaymath}
Also $f$ is the identity. From type $E_7$, we know that 
$\underline{s_{i}}w'_0=w'_0\underline{s_{i}}$, for $i=1,2,3,4,5,6,7$.
Now we need to commute all these $\underline{s_{i}}$ past the
additional factor $u$. After that, the element $\underline{s_{8}}$ has to be 
considered separately. To save space, we only consider the indices.

We start with $\underline{1}$. It commutes with $876534$ 
and changes to $\underline{2}$ passing through $21$.
This changes to  $\underline{3}$ passing through $32$
and then to  $\underline{5}$ passing through $53$.
This $\underline{5}$ commutes with $4$ and 
changes to  $\underline{6}$ passing through $65$.
This $\underline{6}$ commutes with $321$ and 
changes to  $\underline{7}$ passing through $76$.
The $\underline{7}$ commutes with $534235$ and 
changes to  $\underline{6}$ passing through $67$.
This $\underline{6}$ commutes with $8$ and 
changes to  $\underline{7}$ passing through $76$.
The $\underline{7}$ commutes with $5342132534$.
Now we swap $5321$ and $7$ after $6$ as they commute
and after that we move $\underline{7}$ through $67$
changing it to $\underline{6}$.
Now we swap $321$ and $6$ after $5$ as they commute
and after that we move $\underline{6}$ through $56$
changing it to $\underline{5}$.
Now we swap $21$ and $5$ after $3$ as they commute
and after that we move $\underline{5}$ through $35$
changing it to $\underline{3}$.
Next we swap $1$ and $3$ after $2$ as they commute
and after that we move $\underline{3}$ through $23$
changing it to $\underline{2}$.
Next we swap $4$ and $2$ after $1$ as they commute
and after that we move $\underline{2}$ through $12$
changing it to $\underline{1}$. 
It remains to pull $\underline{1}$ on the right
through $435678$. We have checked that 
$\underline{1}$ commutes with $w_0$.

We continue with $\underline{2}$. It commutes with $8765$.
We swap $4$ and $2$ after $3$ and then $\underline{2}$
changes to $\underline{3}$ passing through $32$.
We swap $1$ and $3$ after $4$ and then $\underline{3}$
changes to $\underline{4}$ passing through $43$.
This $\underline{4}$ commutes with $125$ and changes
to  $\underline{3}$ passing through $34$.
This $\underline{3}$ commutes with $6$
and changes to  $\underline{5}$ passing through $53$.
This $\underline{5}$ commutes with $217$
and changes to  $\underline{6}$ passing through $65$.
This $\underline{6}$ commutes with $3423$
and changes to  $\underline{5}$ passing through $56$.
This $\underline{5}$ commutes with $787$
and changes to  $\underline{6}$ passing through $65$.
This $\underline{6}$ commutes with $342132$.
Now we swap $34$ and $6$ after $5$ and move
$\underline{6}$ through $65$ changing it to $\underline{5}$.
Now we swap $5$ and $4$ after $3$ and move
$\underline{5}$ through $35$ changing it to $\underline{3}$.
This $\underline{3}$ changes to  $\underline{4}$ 
passing through $43$.
This $\underline{4}$ commutes with $21765$
and changes to  $\underline{3}$ passing through $34$.
This $\underline{3}$ changes to  $\underline{2}$ 
passing through $23$ and we, finally, pull
it on the right through $5678$. We have checked that 
$\underline{2}$ commutes with $w_0$.

Next we do $\underline{3}$. It commutes with $876$
and changes to  $\underline{5}$  passing through $53$.
This $\underline{5}$ commutes with $421$.
We swap $2$ and $5$ after $3$ and move 
$\underline{5}$  through $35$ changing it to  $\underline{3}$,
which, in turn, changes to $\underline{2}$ 
passing through $23$.
This $\underline{2}$ commutes with $465$
and changes to  $\underline{3}$ passing through $32$.
This $\underline{3}$ commutes with $176$
and changes to  $\underline{5}$ passing through $53$.
This $\underline{5}$ commutes with $42$
and changes to  $\underline{3}$ passing through $35$.
This $\underline{3}$ commutes with $67876$
and changes to  $\underline{5}$ passing through $53$.
This $\underline{5}$ commutes with $421$.
and changes to  $\underline{5}$ passing through $53$.
We swap $2$ and $5$ after $3$ and move 
$\underline{5}$  through $35$ changing it to  $\underline{3}$.
This $\underline{3}$ changes to  $\underline{2}$ passing through $23$.
This $\underline{2}$ commutes with $465$.
and changes to  $\underline{3}$ passing through $32$.
This $\underline{3}$ commutes with $176$.
and changes to  $\underline{5}$ passing through $53$.
This $\underline{5}$ commutes with $42$.
and changes to  $\underline{3}$ passing through $35$.
Now we cal pull $\underline{3}$ out on the right
through $678$. We have checked that 
$\underline{3}$ commutes with $w_0$.

We move to  $\underline{4}$. It commutes with $8765$
and changes to  $\underline{3}$  passing through $34$.
We swap $1$ and $3$ after $2$ and move 
$\underline{3}$  through $23$ changing it to  $\underline{2}$.
This $\underline{2}$ changes to $\underline{1}$ passing through
$12$. This $\underline{1}$ commutes with $534653$
and then changes to $\underline{2}$ passing through
$21$. This $\underline{2}$ commutes with $765$.
We swap $2$ and $4$ after $3$ and move 
$\underline{2}$  through $32$ changing it to  $\underline{3}$.
This $\underline{3}$ changes to $\underline{4}$ passing through
$43$. This $\underline{4}$ commutes with $5678765$
and then changes to $\underline{3}$ passing through $34$. 
We swap $1$ and $3$ after $2$ and move 
$\underline{3}$  through $23$ changing it to  $\underline{2}$.
This $\underline{2}$ changes to $\underline{1}$ passing through
$12$. This $\underline{1}$ commutes with $534653$
and then changes to $\underline{2}$ passing through $21$. 
This $\underline{2}$ commutes with $765$.
We swap $2$ and $4$ after $3$ and move 
$\underline{2}$  through $32$ changing it to  $\underline{3}$.
This $\underline{3}$ changes to $\underline{4}$ passing through
$43$. Finally, we can pull $\underline{4}$ out through
$5678$. We have checked that 
$\underline{4}$ commutes with $w_0$.

Our next element is  $\underline{5}$. It commutes with $87$
and then changes to $\underline{6}$ passing through $65$. 
This $\underline{6}$ commutes with $5342132$.
We swap $34$ and $6$ after $5$ and move 
$\underline{6}$  through $56$ changing it to  $\underline{5}$.
We swap $4$ and $5$ after $3$ and move 
$\underline{5}$  through $35$ changing it to  $\underline{3}$.
This $\underline{3}$ changes to $\underline{4}$ passing through
$43$. This $\underline{4}$ commutes with $21765$
and then changes to $\underline{3}$ passing through
$34$. This $\underline{3}$ changes to $\underline{2}$ passing through
$23$. This $\underline{2}$ commutes with $5678765$.
We swap $4$ and $2$ after $3$ and move 
$\underline{2}$  through $32$ changing it to  $\underline{3}$.
We swap $1$ and $3$ after $4$ and move 
$\underline{3}$  through $43$ changing it to  $\underline{4}$.
This $\underline{4}$ commutes with $125$
and then changes to $\underline{3}$ passing through
$34$. This $\underline{3}$ commutes with $6$
and then changes to $\underline{5}$ passing through
$53$. This $\underline{5}$ commutes with $217$
and then changes to $\underline{6}$ passing through
$65$. This $\underline{6}$ commutes with $3423$
and then changes to $\underline{5}$ passing through
$56$. Finally, we pull $\underline{5}$ out on the right
through $78$. We have checked that 
$\underline{5}$ commutes with $w_0$.

Our next element is  $\underline{6}$. It commutes with $8$
then changes to $\underline{7}$ passing through
$76$. This $\underline{7}$ commutes with $5342132534$.
We swap $5321$ with $7$ after $6$ and then move 
$\underline{7}$ through
$67$ changing it to $\underline{6}$. 
We swap $321$ with $6$ after $5$ and then move 
$\underline{6}$ through
$56$ changing it to $\underline{5}$. 
We swap $21$ with $5$ after $3$ and then move 
$\underline{5}$ through
$35$ changing it to $\underline{3}$. 
We swap $1$ with $3$ after $2$ and then move 
$\underline{3}$ through
$23$ changing it to $\underline{2}$. 
We swap $4$ with $2$ after $1$ and then move 
$\underline{2}$ through
$12$ changing it to $\underline{1}$. 
This $\underline{1}$ commutes with $43567876534$.
and then changes to $\underline{2}$ passing through 
$21$. This $\underline{2}$ changes to $\underline{3}$ 
passing through  $32$.
This $\underline{3}$ changes to $\underline{4}$ 
passing through  $53$.
This $\underline{5}$ commutes with $4$
and then changes to $\underline{6}$ passing through 
$65$. This $\underline{6}$ commutes with $321$
and then changes to $\underline{7}$ passing through 
$76$. This $\underline{7}$ commutes with $534235$
and then changes to $\underline{6}$ passing through 
$67$. Finally, we pull $\underline{6}$ out on the right
through $8$. We have checked that 
$\underline{6}$ commutes with $w_0$.

Our next element is  $\underline{7}$. It changes to
$\underline{8}$ passing through  $87$.
This $\underline{8}$ commutes with $6534213253465321$.
There is a $787$ subword about $10$ letters ahead which
we change to $878$ and then pull the left $8$ all the
way to the left to our coming $7$. Now we can move
$\underline{8}$ through  $78$ changing it to $\underline{7}$.
Next we see a $6786$ subword a few symbols ahead that we change to
$7678$ and pull the left $7$ all the way to the coming $6$.
Now we can move $\underline{7}$ through  $67$ changing 
it to $\underline{6}$.
Next we see a $56785$ subword a few symbols ahead that we change to
$65678$ and pull the left $6$ all the way to the coming $5$.
Now we can move $\underline{6}$ through  $56$ changing 
it to $\underline{5}$.
Next we see a $356783$ subword a few symbols ahead that we change to
$535678$ and pull the left $5$ all the way to the coming $3$.
Now we can move $\underline{5}$ through  $35$ changing 
it to $\underline{3}$. Now we swap $2$ and $4$.
and move the second occurrence of $4$ to the left
hitting $3$. This produces a $434$ subword that we
change to $343$. Now we can move $\underline{3}$
through $23$ changing it to $\underline{2}$
which, in turn, commutes with $4$. Now we
swap $5678$ with $2$ after $3$ and move 
$\underline{2}$ through $32$ changing it to $\underline{3}$.
Now we swap $6781$ with $3$ after $5$ and move 
$\underline{3}$ through $53$ changing it to $\underline{5}$.
Now we swap $7812$ with $5$ after $6$ and move 
$\underline{5}$ through $65$ changing it to $\underline{6}$.
Now we swap $81234$ with $6$ after $7$ and move 
$\underline{6}$ through $76$ changing it to $\underline{7}$.
Now we swap $12345321$ with $7$ after $8$ and move 
$\underline{7}$ through $87$ changing it to $\underline{8}$.
This $\underline{8}$ commutes with $1234532165342356$
and then changes to $\underline{7}$ passing through
$78$. We have checked that 
$\underline{7}$ commutes with $w_0$.

It remains to check $\underline{s_{8}}$. Let
$w''_0$ be the longest element for type $E_6$.
Then we have $\underline{s_{8}}w''_0=w''_0\underline{s_{8}}$
due to commutativity. Now we have to move 
$\underline{s_{8}}$ through the product of the additional
factor in type $E_7$ followed by the additional factor in
type $E_8$. Again, we use SageMath for computation of 
reduced words and left descent sets.

The product of two additional factors is given by the following:

\resizebox{\textwidth}{!}{
$
765342132534653217653423567
876534213253465321765342356787653421325346532176534235678
$
}

which we rewrite as

\resizebox{\textwidth}{!}{
$
765342132534653217653423567
876534213253465321765342356787653421325346532176534235678
$
}

Here the left descent set is $\{7,8\}$, so we can 
move $\underline{8}$ through the prefix 
$878$ changing it to $\underline{7}$ and the remaining
suffix is

\resizebox{\textwidth}{!}{
$
653423567123564352341235678765342356712356435
234123567865342356712356435234123567
$
} 

This has left descent set $\{6\}$ and  $67$ is a prefix, 
so we can  move $\underline{7}$ through this prefix 
changing it to $\underline{6}$ and the remaining
suffix is

\resizebox{\textwidth}{!}{
$
5342356712356435234123567876534235671235643523412356785342356712356435234123567
$
}

This has left descent set $\{5\}$ and  $56$ is a prefix, 
so we can  move $\underline{6}$ through this prefix 
changing it to $\underline{5}$ and the remaining
suffix is

\resizebox{\textwidth}{!}{
$
34235671235643523412356787653423567123564352341235678342356712356435234123567
$
}

This has left descent set $\{3\}$ and  $35$ is a prefix, 
so we can  move $\underline{5}$ through this prefix 
changing it to $\underline{3}$ and the remaining
suffix is

\resizebox{\textwidth}{!}{
$
423567123564352341235678765342356712356435234123567842356712356435234123567
$
}

This has left descent set $\{2,4\}$ and  $43$ is a prefix, 
so we can  move $\underline{3}$ through this prefix 
changing it to $\underline{4}$ and the remaining
suffix is

\resizebox{\textwidth}{!}{
$
2356712356435234123567876534235671235643523412356782356712356435234123567
$
}

This has left descent set $\{2\}$ and $\underline{4}$ commutes
with it, so we just remove this $2$. In what remains, 
the left descent set is $\{1,3\}$ and $\underline{4}$ commutes
with $1$, so we just remove this $1$ and the following suffix remains:
\begin{displaymath}
35672356435234123567876534235671235643523412356782356712356435234123567 
\end{displaymath}
This has left descent set $\{3\}$ and  $34$ is a prefix, 
so we can  move $\underline{4}$ through this prefix 
changing it to $\underline{3}$ and the remaining
suffix is
\begin{displaymath}
567235643523412356787653423567123564352341235678356712356435234123567
\end{displaymath}
This has left descent set $\{2,5\}$ and  $23$ is a prefix, 
so we can  move $\underline{3}$ through this prefix 
changing it to $\underline{2}$ and the remaining
suffix is
\begin{displaymath}
5673564352341235678765342356712356435234123567835672356435234123567
\end{displaymath}
This has left descent set $\{5\}$ and $\underline{2}$ commutes
with it, so we just remove it. In the remainder,
the left descent set is $\{3,6\}$ and $\underline{2}$ commutes
with $6$, so we just remove it. In the remainder,
the left descent set contains $7$ and $\underline{2}$ commutes
with $7$, so we just remove it. In the remainder,
the left descent set contains $8$ and $\underline{2}$ commutes
with $8$, so we just remove it. The remaining suffix is
\begin{displaymath}
356435234123567865342356712356435234123567835672356435234123567
\end{displaymath}
This has left descent set $\{3\}$ and  $32$ is a prefix, 
so we can  move $\underline{2}$ through this prefix 
changing it to $\underline{3}$ and the remaining
suffix is
\begin{displaymath}
5643523412356786534235671235643523412356785672356435234123567
\end{displaymath}
This has left descent set $\{4,5\}$ and  $43$ is a prefix, 
so we can  move $\underline{3}$ through this prefix 
changing it to $\underline{4}$ and the remaining
suffix is
\begin{displaymath}
56352341235678653423567123564352341235678567356435234123567
\end{displaymath}
We remove the first $5$ as it commutes with $\underline{4}$.
What remains has the left descent set $\{3,6\}$ and
$\underline{4}$ commutes with $6$, so we remove it.
What remains has $7$ in the left descent set and
$\underline{4}$ commutes with $7$, so we remove it.
The remaining suffix is
\begin{displaymath}
35234123567853423567123564352341235678567356435234123567
\end{displaymath}
and it has left descent set $\{3\}$ with $34$ being a prefix.
so we can  move $\underline{4}$ through this prefix 
changing it to $\underline{3}$ and the remaining
suffix is
\begin{displaymath}
523412356785342356712356435234123567867356435234123567
\end{displaymath}
This has left descent set $\{2,5\}$ with $23$ being a prefix.
so we can  move $\underline{3}$ through this prefix 
changing it to $\underline{2}$ and the remaining
suffix is
\begin{displaymath}
5341235678534235671235643523412356786756435234123567
\end{displaymath}
As $\underline{2}$ commutes with $5$, we delete $5$.
What remains has $6$ in the left descent set and
$\underline{2}$ commutes with $6$, so we remove this $6$.
The remainder has left descent set $\{1,3\}$ with $12$ being a prefix.
so we can  move $\underline{2}$ through this prefix 
changing it to $\underline{1}$ and the remaining
suffix is
\begin{displaymath}
342356783423567123564352341235678675635234123567
\end{displaymath}
Here we can move $\underline{1}$ past $34$. 
The remainder has $5$ in the left descent set and
$\underline{1}$ commutes with $5$, so we remove this $5$.
The next remainder has $3$ in the left descent set and
$\underline{1}$ commutes with $3$, so we remove this $3$.
The remainder has $\{2\}$ as the left descent set an
$21$ as a prefix. We move $\underline{1}$ through this prefix 
changing it to $\underline{2}$ and the remaining
suffix is
\begin{displaymath}
356784356723564352341235678675635234123567
\end{displaymath}
This has $\{3\}$ as the left descent set an
$32$ as a prefix. We move $\underline{2}$ through this prefix 
changing it to $\underline{3}$ and the remaining
suffix is
\begin{displaymath}
5678435672356435234123567875635234123567
\end{displaymath}
This has $\{4,5\}$ as the left descent set an
$43$ as a prefix. We move $\underline{3}$ through this prefix 
changing it to $\underline{4}$ and the remaining
suffix is
\begin{displaymath}
56783567235643523412356785635234123567
\end{displaymath}
As $\underline{4}$ commutes with $5678$, we just remove the latter.
The remainder has descent set $\{3\}$ and
$34$ as a prefix.  We move $\underline{4}$ through this prefix 
changing it to $\underline{3}$ and the remaining
suffix is
\begin{displaymath}
56723564352341235678635234123567
\end{displaymath}
This has $\{2,5\}$ as the left descent set an
$23$ as a prefix. We move $\underline{3}$ through this prefix 
changing it to $\underline{2}$ and the remaining
suffix is
\begin{displaymath}
567356435234123567835234123567
\end{displaymath}
Here $\underline{2}$ commutes with $567$ so we remove this.
The remainder has $\{3\}$ as the left descent set an
$32$ as a prefix. We move $\underline{2}$ through this prefix 
changing it to $\underline{3}$ and the remaining
suffix is
\begin{displaymath}
5643523412356785234123567
\end{displaymath}
This has $\{4,5\}$ as the left descent set an
$43$ as a prefix. We move $\underline{3}$ through this prefix 
changing it to $\underline{4}$ and the remaining
suffix is
\begin{displaymath}
56352341235678234123567
\end{displaymath}
Here $\underline{4}$ commutes with $56$ so we remove this part.
What remains  has $\{3\}$ as the left descent set an
$34$ as a prefix. We move $\underline{4}$ through this prefix 
changing it to $\underline{3}$ and the remaining
suffix is
\begin{displaymath}
5234123567834123567
\end{displaymath}
This  has $\{2,5\}$ as the left descent set and
$53$ as a prefix. We move $\underline{3}$ through this prefix 
changing it to $\underline{5}$ and the remaining
suffix is
\begin{displaymath}
23412356783423567
\end{displaymath}
Here $2$ commutes with $\underline{5}$, so we delete it.
The remainder has $\{1,3\}$ as the left descent set and
$1$ commutes with $\underline{5}$, so we remove this $1$
resulting in 
\begin{displaymath}
342356783423567
\end{displaymath}
This has $\{3\}$ as the left descent set and prefix 
$35$. We move $\underline{5}$ through this prefix 
changing it to $\underline{3}$ and the remaining
suffix is
\begin{displaymath}
4235678423567
\end{displaymath}
This has $\{2,4\}$ as the left descent set and prefix 
$43$. We move $\underline{3}$ through this prefix 
changing it to $\underline{4}$ and the remaining
suffix is
\begin{displaymath}
23567843567
\end{displaymath}
Here $\underline{4}$ commutes with $2$ so we remove
the latter. The remainder has $\{3\}$ 
as the left descent set and prefix 
$34$. We move $\underline{4}$ through this prefix 
changing it to $\underline{3}$ and the remaining
suffix is $56783567$. We rewrite this as $53657687$
and now we move $\underline{3}$ past $53$
changing it to $\underline{5}$ then past $65$
changing it to $\underline{6}$ then past $76$
changing it to $\underline{7}$ and, finally,  past $87$
changing it to $\underline{8}$. We are done!

\subsection{Proof of Lemma~\ref{lem-tilt2-1}}\label{s5.9}

By \cite[Theorem~7.3]{J-Z}, sending $s$ to $\mathcal{L}\mathbf{C}_{s}$
and  $\underline{s}$ to $\theta_s$, for $s\in S$, defines a weak
action of $SB$ on $\mathcal{O}_0$. Consequently, 
the result of Subsection~\ref{s5.2} implies that 
\begin{displaymath}
\theta_s\circ \mathcal{L}\mathbf{C}_{w_0^\mathfrak{p}} \cong
\mathcal{L}\mathbf{C}_{w_0^\mathfrak{p}}\circ \theta_{f(s)}.
\end{displaymath}
The claim of Lemma~\ref{lem-tilt2-1} follows.


\noindent
V.~M.: Department of Mathematics, Uppsala University, Box. 480,
SE-75106, Uppsala, SWEDEN, email: {\tt mazor\symbol{64}math.uu.se}

\noindent
S.~S.: Department of Mathematics, 
Indian Institute of Technology, Dharwad, 580011,
INDIA, email: {\tt maths.shraddha\symbol{64}gmail.com}

\end{document}